\documentclass[smallextended,envcountsect,envcountsame]{svjour3}       
\smartqed  
\usepackage{graphicx}

\usepackage{geometry}
\usepackage[misc]{ifsym}

\usepackage{amsmath}

\usepackage{comment}
\usepackage{amsfonts}
\usepackage{amssymb}
\usepackage{graphicx} 
\usepackage{enumitem}
\usepackage{mathtools}
\usepackage{float}
\usepackage[center]{caption}

\usepackage{xcolor}

\usepackage{soul}

\usepackage{thmtools} 

\usepackage[nameinlink]{cleveref} 

\newtheorem{prop}[theorem]{Proposition}%
\newtheorem{assumption}[theorem]{Assumption}%
\newtheorem{cor}[theorem]{Corollary}

\renewcommand{\d}{\, \mathrm{d}}
\renewcommand{\L}{\mathrm{L}} 
\renewcommand{\H}{\mathrm{H}} 
\newcommand{\e}{\mathrm{e}} 

\journalname{}
\begin{document}

\title{
BIBO stability for funnel control: semilinear internal dynamics with unbounded input and output operators
\thanks{
This research was partially conducted with the financial support of F.R.S-FNRS, Belgium. A. H. is a FNRS Research Fellow under the grant CR 40010909.\newline
The second and third named authors are supported by the German Research Foundation (DFG) via the joint grant JA 735/18-1 / SCHW 2022/2-1.}
}

\author{Anthony Hastir         \and
Ren\'e Hosfeld \and 
        Felix L.~Schwenninger     \and
        Alexander A.~Wierzba 
}

\institute{\Letter\,\, A. Hastir \at
              Department of mathematics and Namur Institute for Complex Systems (naXys), University of Namur, Rue de Bruxelles, 61, Namur, B-5000, Belgium\\
              \email{anthony.hastir@unamur.be}
           \and
              R. Hosfeld \at
              University of Wuppertal, School of Mathematics and Natural Sciences,
              IMACM, Gaußstraße 20, D-42119 Wuppertal, Germany\\
              \email{hosfeld@uni-wuppertal.de}
           \and
           F.L. Schwenninger, A.A. Wierzba \at
              Department of Applied Mathematics, University of Twente, P.O. Box 217, Enschede, 7500, The Netherlands\\
              \email{f.l.schwenninger@utwente.nl, a.a.wierzba@utwente.nl}
}

\date{}

\maketitle

\begin{abstract}
This note deals with Bounded-Input-Bounded-Output (BIBO) stability for semilinear infinite-dimensional dynamical systems allowing for boundary control and boundary observation. We give sufficient conditions that guarantee BIBO stability based on Lipschitz conditions with respect to interpolation spaces. Our results can be applied to guarantee feasibility of funnel control for coupled ODE-PDE systems, as shown by means of an example from  chemical engineering.

\end{abstract}


\section{Introduction}
In this note we report recent progress in the adaptive control of certain classes of nonlinear, infinite-dimensional systems. This is done by studying the interplay of two relatively well-studied concepts from finite-dimensional theory; \emph{funnel control} and \emph{Bounded-Input-Bounded-Output stability (BIBO stability)}. Whereas the latter is classical in the history of systems theory, funnel control has only been established in the last twenty years starting from seminal work by Ilchmann--Sangwin--Ryan \cite{Ilchmann2002}, see the survey \cite{Berger2020funnel} and the references therein. For both infinite-dimensional theory exists only partially, in particular in the nonlinear case. 

Funnel control has shown to be intimately connected to BIBO stability of the internal dynamics if the system allows for a relative degree, see~\cite{Ilchmann2002,BergerFunnelAutomatica,BergerPucheSchwenninger} for the evolution of such ``Funnel Theorems'' and their applications to both finite and infinite-dimensional systems. From an analytic point of view, this relative degree result is an input-output behaviour given by an ordinary differential equation coupled with a possibly infinite-dimensional system. In our setting, the latter \emph{internal dynamics} are represented by a semilinear partial differential equation with various assumptions on the nonlinearity. The power of the above mentioned ``Funnel Theorems'' lies in the fact that BIBO stability of the internal dynamics is essentially sufficient for guaranteeing that funnel control works.

The funnel controller is an adaptive model-free output error feedback whose objective is to let the output of a dynamical system follow a predetermined (time-varying) reference signal. In contrast to classical tracking, this objective is fulfilled in the sense that the output error has to remain within prescribed funnel boundaries. The funnel controller is adaptive in the sense that the time-varying gain function adapts to the current value of the output error. This field of adaptive control has attracted a lot of attention in the last few years. For systems with relative degree one, funnel control is extensively developed in \cite{Ilchmann2002} and \cite{Ilchmann2005}. Few years later, funnel control has been applied to MIMO systems with known strict relative degree, see~\cite{Berger2021}. The different types of dynamical systems for which funnel control is conducive are listed in \cite{Berger2021Bis}.

BIBO stability of the internal dynamics is the crucial property when proving that funnel control indeed can be applied successfully, provided the system under investigation is of relative degree form. However, checking for BIBO stability for infinite-dimensional systems turns out to be subtle: In \cite{Berger2020funnel} a funnel controller was developed to regulate the position of a moving water tank, with internal dynamics being approximated by a linear wave equation. The main ingredient of that proof rests on carefully investigating the transfer function and showing that its inverse Laplace transform exists as measure of bounded total variation, see also \cite{BergerPucheSchwenninger} and, also the Callier--Desoer class, \cite{CurtainZwartNew}. Indeed, for linear systems---including boundary control and observation---with finite-dimensional input and output spaces, this property characterizes BIBO stability, \cite{SWZ_BIBOStability}. A class of systems with semilinear internal dynamics was considered in \cite{Hastir_Funnel}, under additional assumptions on the linear part and global Lipschitz continuity of the nonlinearity. This extended earlier works \cite{IlchmannByrnesIsidori} for linear systems involving a Byrnes--Isidori normal form. However, a drawback of this approach is that boundary control and observation seems to be excluded, which was one of the starting points of the present paper.

We point out that funnel control has found application in a much wider context, but we focus our presentation to the theory relevant for our results.
In particular, Funnel control was recently applied to general systems given by partial differential equations, which are not of the form allowing for a relative degree, see \cite{Berger2021} and \cite{PucheSchwenningerReis}. Note also that funnel control has been lately coupled to model-predictive-control (Funnel MPC) for nonlinear systems with relative degree one, see~\cite{Berger2021funnel}.

In this paper, we use $\L^{\infty}$-BIBO stability to treat semilinear internal dynamics allowing for unbounded input and output operators. This result is then used to conclude feasibility of funnel control. As the latter is a direct application of known Funnel Theorems, the main contribution of the paper lies in carefully exploring BIBO stability for such systems. To do so,  we follow the idea of writing the nonlinear state space system as the interconnection of an extended linear system and a nonlinear feedback. The main result consists in showing that, provided that the nonlinear operator satisfies some global Lipschitz condition, the whole nonlinear system is $\L^\infty$-BIBO stable. To apply our results, we highlight the fact that the linear part of the dynamics has to satisfy some regularity assumptions, which covers a class of applications driven by parabolic partial differential equations.

The paper is organized as follows: \Cref{BIBO_Section} aims at defining the notion of $\L^\infty$-BIBO stability for nonlinear infinite-dimensional systems with general input and output operators. A way of deriving BIBO stability of an extended linear system constructed as the original linear system with augmented input and output operators is detailed in \Cref{Section_Interconnection}. BIBO stability of the nonlinear system is shown in \Cref{Section_NonlinearOperator} in the case where the nonlinear operator is globally Lipschitz, viewed as an operator defined on a more regular space than the state space, constructed as the domain of the fractional powers of the opposite of the linear operator dynamics. An example of a heat equation with a locally Lipschitz nonlinearity and internal control is tackled in \Cref{Section_Local_Lipschitz}, which requires a different approach than  in \Cref{Section_NonlinearOperator}. 
Funnel control is applied in \Cref{Section_Application_Chemical} to a nonlinear convection-reaction-diffusion system coupled to a linear ordinary differential equation. The needed assumption of BIBO stability of the internal dynamics is shown thanks to our main results. Conclusions and perspectives are given in \Cref{Section_CCL}.

\section{BIBO stability of semi-linear state space systems}\label{BIBO_Section}
In the following let $U,X,Y$ refer to Banach spaces and  by $\mathcal{L}(X,Y)$ we denote the bounded linear operators from $X$ to $Y$.
Recall the notion of a system node going back to \cite{Staffans}. A quadruple $\Sigma(A,B,C,\mathbf{G})$ is called a \emph{system node} for $(U,X,Y)$ if $A$ is the generator of a strongly continuous semigroup $(T(t))_{t\geq 0}$ of linear operators on $X$, $B\in\mathcal{L}(U,X_{-1})$, $C\in \mathcal{L}(X_{1},Y)$ and $\mathbf{G}:\mathbb{C}_{\omega_{0}}\to \mathcal{L}(U,Y)$ is a transfer function, i.e.\ $\mathbf{G}$ is a holomorphic function satisfying the equation 
\begin{equation*}
    \mathbf{G}(s)-\mathbf{G}(t)=C\left[(sI-A)^{-1}-(tI-A)^{-1}\right]B,\qquad s,t\in\mathbb{C}_{\omega_0},
\end{equation*}
where $\mathbb{C}_{\omega_{0}}=\{z\in\mathbb{C}\colon \Re(z)>\omega_0\}$ and $\omega_0$ equals the growth bound of the semigroup. The space $X_{1}$ refers to the domain $D(A)$ of $A$ equipped with the graph norm, whereas $X_{-1}$ denotes the completion of $X$ with respect to the norm $\|(\beta I-A)^{-1}\cdot\|_X$ for some $\beta\in\rho(A)$. For details on this concept we refer to \cite{Staffans}. 
We use the following notion of a semilinear state space system and its mild solutions.
\begin{definition}[Semilinear state space system]\label{def:semilinear_state_space_system}
    Let $\Sigma(A,B,C,\mathbf{G})$ be a linear system node with state space $X$, semigroup generator $A: D(A) \subset X \rightarrow X$, input space $U$, input operator $B: U \rightarrow X_{-1}$, output space $Y$ and output operator $C: X_1 \rightarrow Y$. Let $C\&D: D(C\&D) \rightarrow Y$ be the associated combined output/feedthrough operator \cite{Staffans,TucsnakWeiss2014} given by
    \begin{align*}
        D(C\&D)={}&\left\{\begin{bmatrix}x\\u\end{bmatrix}\in X\times U \colon Ax+Bu\in X \right\}\\
        C\&D\begin{bmatrix}x\\u\end{bmatrix}={}&C\left[x-(sI-A)^{-1}Bu\right]+\mathbf{G}(s)u,
    \end{align*}
    for some fixed $s\in\mathbb{C}_{\omega_{0}}$.
    Furthermore, let $f: Z \rightarrow X$ be a nonlinear function with $Z \subset X$ a continuously embedded subspace. Then the pair $(\Sigma, f)$ formally representing the equations
     \begin{equation}
     \label{eq:semilinStateSpaceNode}
         \left\{
         \begin{array}{l}
             \Dot{x}(t) = A x(t) + B u(t) + f(x(t)), \qquad x(0)=x_{0}, \\
             y(t) = C\&D \begin{bmatrix} x(t) \\ u(t) \end{bmatrix},
         \end{array}\right.
     \end{equation}
    where $t>0$, is called a \emph{semilinear state space system}.
\end{definition}

\begin{remark}
    Here, the space $Z$ will usually be some interpolation space $X_1 \subset X_\alpha \subset X$ with $0 \leq \alpha \leq 1$.
\end{remark}

\begin{definition}[Mild solution of a semilinear state space system]\label{def:mildSolution}
    Let $x_0 \in X$, $T > 0$ and $u \in \L^1_\textrm{loc}\left([0,T]; U\right)$. A triple $(u,x,y)$ is called a \emph{mild solution} to the semilinear state space system \eqref{eq:semilinStateSpaceNode} on $[0,T]$ with initial value $x_0$  if
    \begin{enumerate}
        \item $x:[0,T] \rightarrow X_{-1}$ and $x(t) \in Z$ for almost all $t \in[0,T]$ and $f(x(\cdot))\in \L^1_\textrm{loc}\left([0,T]; X\right)$ 

        \item $x$ solves 
        \begin{equation*}
            x(t) = T(t) x_0 + \int_0^t T(t-s) \left[ f(x(s)) + B u(s) \right] \d s
        \end{equation*}
        in $X_{-1}$ for all $t \in [0,T]$;
        \item $y$ is a $Y$-valued distribution given by 
        \begin{equation}\label{eq:outputDistExpression}
            y(t) = \frac{\textrm{d}^2}{\textrm{d}t^2}  \left( \left( C\&D \right) \int_0^t (t - s) \begin{bmatrix} x(s) \\ u(s) \end{bmatrix} \d s \right),
        \end{equation}
        meaning that it acts on test functions $\varphi \in \mathcal{C}^\infty_c(\mathbb{R}_{\geq 0};Y')$ as
        \begin{equation*}
        y[\varphi] = \int_0^\infty \left\langle \varphi''(t) , \left( C\&D \right) \int_0^t (t - s) \begin{bmatrix} x(s) \\ u(s) \end{bmatrix} \d s \right\rangle_{Y', Y} \d t.
    \end{equation*}
    \end{enumerate}
    A triple $(u,x,y)$ is called a \emph{global mild solution} to the semilinear state space system, if $(u\mid_{[0,T]},x\mid_{[0,T]},y\mid_{[0,T]})$ is a mild solution on $[0,T]$ for every $T > 0$. If $u$ and $y$ are clear from the context, e.g.\ if $y=x$, we simply refer to $x$ as the mild solution $(u,x,y)$.
\end{definition}
\begin{remark}
    Part 3 of \Cref{def:mildSolution} is equivalent to the statement, that $\left( \left[ \begin{smallmatrix}
            u \\ f(x)
        \end{smallmatrix} \right], x, \left[ \begin{smallmatrix}
            y \\ x
        \end{smallmatrix} \right] \right)$ is a generalised solution of the extended system node 
        \begin{equation*}
            \Sigma\left(A, \left[ \begin{smallmatrix} B & I \end{smallmatrix} \right], \left[ \begin{smallmatrix} C \\ I \end{smallmatrix} \right], \left[ \begin{smallmatrix} \mathbf{G} & C (\cdot I - A)^{-1} \\ (\cdot I - A)^{-1} B & (\cdot I - A)^{-1} \end{smallmatrix} \right] \right)
        \end{equation*}
        as defined in \cite[Definition~4.7.10]{Staffans}. 
        Note that this also guarantees that the integral appearing in Equation \eqref{eq:outputDistExpression} indeed always lies in $D(C\&D)$ and thus the application of $C\&D$ is well-defined \cite[Lemma~4.7.9]{Staffans}.
\end{remark}
With this solution concept we can define BIBO stability for the considered semilinear state space systems.
\begin{definition}[$\L^\infty$-BIBO stability]\label{def:LInfty_BIBO}
    Let $(\Sigma, f)$ be a semilinear state space system. Then it is called \emph{$\L^\infty$-BIBO stable} if the following two conditions are satisfied.
    \begin{enumerate}
        \item for any $u \in \L^\infty_{\text{loc}}(\mathbb{R}_{\geq 0}; U)$ there exists a global mild solution $(u,x,y)$ with $x_0 = 0$ and $y \in \L^\infty_\textrm{loc}\left(\mathbb{R}_{\geq 0}; Y \right)$,
        \item for any $c_U > 0$ there exists a constant $c_Y > 0$ such that for any $t > 0$ and any global mild solution $(u,x,y)$ of $(\Sigma, f)$ with $x_0 = 0$, $u \in \L^\infty_{\text{loc}}(\mathbb{R}_{\geq 0}; U)$, the following implication holds:
        \begin{equation*}
        \| u \|_{\L^\infty([0,t];U)} < c_U \quad\implies\quad\| y \|_{\L^\infty([0,t];Y)} < c_Y.
        \end{equation*}
    \end{enumerate}
\end{definition}
We recall the following definition of admissible control operators \cite{Weiss89ControlOperators}. 
\begin{definition}[Admissible control operators]
    Let $A: D(A) \subset X \rightarrow X$ be the generator of a strongly continuous semigroup $(T(t))_{t \geq 0}$ and $1 \leq p \leq \infty$. Then $B \in \mathcal{L}(U, X_{-1})$ is called a \emph{$\L^p$-admissible control operator} (or just \emph{$\L^p$-admissible}) if for some (and hence for all) $t > 0$ there exists a (minimal) constant $C_t>0$ such that
    \begin{equation*}
    \left\lVert \int_0^t T(t-s) Bu(s) \d s \right\rVert_X \leq C_t \lVert u \rVert_{\L^p([0,t];U)}
    \end{equation*}
    holds for all $u \in \L^p([0,t]; U)$. If $C_\infty \coloneqq \sup_{t >0}C_t< \infty$, then $B$ is called \emph{infinite-time $\L^p$-admissible}.
\end{definition}

\begin{remark}
In the above definition, we implicitly assumed that for some (and hence for all) $t>0$ and all $u \in \L^p(\mathbb{R}_{\geq 0};U)$ it holds that
 \begin{equation*}
    \int_0^t T(t-s) B u(s) \d s \in X,
\end{equation*}
which is even sufficient for $\L^p$-admissibility of $B$, see for instance \cite[Proposition~4.2]{Weiss89ControlOperators}. Moreover, it can be concluded from the proof of \cite[Proposition~2.5]{Weiss89ControlOperators} that, if $(T(t))_{t \geq 0}$ is exponentially stable, i.e. there exists $M \geq 1$ and $\omega>0$ such that
$\lVert T(t) \rVert \leq M \e^{- \omega t}$ for every $t \geq 0$, then $\L^p$-admissibility and infinite-time $\L^p$-admissibility of $B$ are equivalent.
\end{remark}

\section{Linear systems with nonlinear feedback and BIBO stability}\label{Section_Interconnection}
A way of approaching the question of BIBO stability for systems written as in \eqref{eq:semilinStateSpaceNode} is explained in this section together with some preliminary result. 
One possible approach to study semilinear systems of the form of Equation \eqref{eq:semilinStateSpaceNode} that was for instance used in \cite{HastirSCL}, is to rewrite the system in the form as depicted in \Cref{fig:FeedbackLoop} and consider the nonlinearity as a nonlinear feedback loop attached to an extended linear system.

This way it is possible to employ properties of the linear system to derive properties of the semilinear one. Here the most relevant property of the linear system for our discussions is naturally its $\L^\infty$-BIBO stability.

The following proposition provides a sufficient condition for when this is the case.

\begin{figure}[h!]
\includegraphics[scale=1]{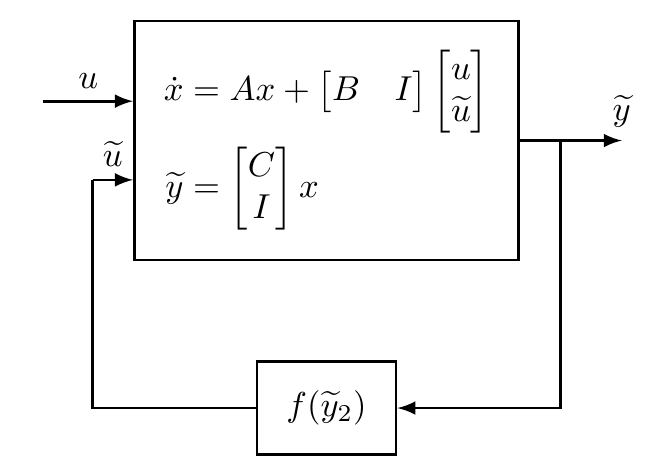}
\centering
\caption{Nonlinearity as feedback loop}
\label{fig:FeedbackLoop}
\end{figure}

\begin{prop}
\label{prop:extendedLinSystemBIBO}
Let $\Sigma(A,B,C,\mathbf{G})$ be a system node with $A$ the generator of an exponentially stable semigroup. 
Then the extended system node $\Sigma\left(A, \left[ \begin{smallmatrix} B & I \end{smallmatrix} \right], \left[ \begin{smallmatrix} C \\ I \end{smallmatrix} \right], \left[ \begin{smallmatrix} \mathbf{G} & C (\cdot I - A)^{-1} \\ (\cdot I - A)^{-1} B & (\cdot I - A)^{-1} \end{smallmatrix} \right] \right)$ is $\L^\infty$-BIBO stable if all of the following hold:
\begin{itemize}
    \item $\Sigma(A,B,C,\mathbf{G})$ is $\L^\infty$-BIBO stable.
    \item $B$ is an $\L^\infty$-admissible control operator.
    \item $\Sigma(A,I,C,C (\cdot I - A)^{-1})$ is $\L^\infty$-BIBO stable.
\end{itemize}
\end{prop}
\begin{proof}

By \cite[Corollary~6.2]{SWZ_BIBOStability}, from the $\L^\infty$-admissibility of $B$ together with the exponential stability, it follows that the system node  $\Sigma(A,B,I,(\cdot I - A)^{-1} B)$ is $\L^\infty$-BIBO stable. Analogously the same follows for the system node $\Sigma(A,I,I,(\cdot I - A)^{-1})$.

By the $\L^\infty$-BIBO stability of the respective system nodes it follows that there are constants $c, c_C, c_B, c_I > 0$ such that for any $\left[ \begin{smallmatrix} u \\ \widetilde{u} \end{smallmatrix} \right] \in \L_{\rm loc}^\infty(\mathbb{R}_{\geq 0}; U \times X)$ and $x_0 = 0$  there are the following solutions which then, for any $t > 0$, satisfy the corresponding inequalities:
\begin{itemize}
    \item $(u, x, y)$ of $\Sigma(A,B,C,\mathbf{G})$ with $y \in \L_{\rm loc}^\infty(\mathbb{R}_{\geq 0}; Y)$ and $\| y \|_{\L^\infty([0,t]; Y)} \leq c \| u \|_{\L^\infty([0,t]; U)}$,
    \item $(u, x_B, x_B)$ of $\Sigma(A,B,I,(\cdot I - A)^{-1} B)$ with $x_B \in \L_{\rm loc}^\infty(\mathbb{R}_{\geq 0}; X)$ and $\| x_B \|_{\L^\infty([0,t]; X)} \leq c_B \| u \|_{\L^\infty([0,t]; U)}$,
    \item $(\widetilde{u}, x_C, y_C)$ of $\Sigma(A,I,C,C (\cdot I - A)^{-1})$ with $y_C \in \L_{\rm loc}^\infty(\mathbb{R}_{\geq 0}; Y)$ and $\| y_C \|_{\L^\infty([0,t]; Y)} \leq c_C \| \widetilde{u} \|_{\L^\infty([0,t]; X)}$,
    \item $(\widetilde{u}, x_I, x_I)$ of $\Sigma(A,I,I,(\cdot I - A)^{-1})$ with $x_I \in \L_{\rm loc}^\infty(\mathbb{R}_{\geq 0}; Y)$ and $\| x_I \|_{\L^\infty([0,t]; X)} \leq c_I \| \widetilde{u} \|_{\L^\infty([0,t]; X)}$.
\end{itemize}
Clearly we also have $x_B = x$ and $x_C = x_I$.

Then we find for the extended system that the state corresponding to the input $\left[ \begin{smallmatrix} u \\ \widetilde{u} \end{smallmatrix} \right]$ takes the form
\begin{equation*}
    \widetilde{x}(t) = \int_0^t T(t-s) \begin{bmatrix} B & I \end{bmatrix} \begin{bmatrix} u(s) \\ \widetilde{u}(s) \end{bmatrix} \d s = x(t) + x_C(t).
\end{equation*}
We furthermore observe that the combined output/feedthrough operator $\widetilde{C\&D}$ for this extended system node acts as
\begin{equation*}
\begin{split}
    \widetilde{C\&D} \begin{bmatrix} \widetilde{x} \\ \begin{bmatrix} u \\ \widetilde{u} \end{bmatrix} \end{bmatrix} &= \begin{bmatrix} C \\ I \end{bmatrix} \left( \widetilde{x} - (\beta I - A)^{-1} \begin{bmatrix} B & I\end{bmatrix} \begin{bmatrix} u \\ \widetilde{u} \end{bmatrix} \right) + \left[ \begin{smallmatrix} \mathbf{G}(\beta) & C (\beta I - A)^{-1} \\ (\beta I - A)^{-1} B & (\beta I - A)^{-1} \end{smallmatrix} \right] \begin{bmatrix} u \\ \widetilde{u} \end{bmatrix} \\
     &= \begin{bmatrix} C \\ I \end{bmatrix} \left( x - (\beta I - A)^{-1} B u \right)
     + \begin{bmatrix} C \\ I \end{bmatrix} \left( x_I - (\beta I - A)^{-1}  \widetilde{u} \right)
     + \left[ \begin{smallmatrix} \mathbf{G}(\beta) u + C (\beta I - A)^{-1} \widetilde{u} \\ (\beta I - A)^{-1} B u + (\beta I - A)^{-1} \widetilde{u} \end{smallmatrix} \right] \\
    &= \begin{bmatrix} y + y_C \\  x + x_C \end{bmatrix}
\end{split}
\end{equation*}
and we thus observe that the output $\widetilde{y}$ of the extended system node is given by
\begin{equation*}
    \widetilde{y} = \begin{bmatrix} y + y_C \\ x_B + x_I \end{bmatrix}
\end{equation*}
a-priori in a distributional sense, but thus also as $\widetilde{y} \in \L_{\rm loc}^\infty(\mathbb{R}_{\geq 0}; Y \times X)$.  This shows the existence of a solution $(\left[ \begin{smallmatrix} u \\ \widetilde{u} \end{smallmatrix} \right], \widetilde{x}, \widetilde{y})$ with $\widetilde{y} \in \L_{\rm loc}^\infty(\mathbb{R}_{\geq 0}; Y \times X)$ of the extended system node.

Furthermore we then find for all $t > 0$
\begin{equation*}
    \begin{split}
        \| \widetilde{y} \|_{\L^\infty([0,t]; Y \times X)} &\lesssim 
        \| y \|_{\L^\infty([0,t]; Y)} + 
        \| y_C \|_{\L^\infty([0,t]; Y)} +
        \| x_B \|_{\L^\infty([0,t]; X)} +
        \| x_I \|_{\L^\infty([0,t]; X)} \\
        &\leq c \| u \|_{\L^\infty([0,t]; U)} + 
        c_C \| u \|_{\L^\infty([0,t]; U)} +
        c_B \| \widetilde{u} \|_{\L^\infty([0,t]; X)} +
        c_I \| \widetilde{u} \|_{\L^\infty([0,t]; X)} \\
        &\lesssim \left\| \begin{bmatrix}
            u \\ \widetilde{u}
        \end{bmatrix} \right\|_{\L^\infty([0,t]; U \times X)}
    \end{split}
\end{equation*}
which finishes the proof. \qed
\end{proof}

\begin{remark}\label{rem:extendedlinearsys}
    \begin{enumerate}
        \item In the following we will use the notation $\Sigma(A,B,C)$ to refer to a system node $\Sigma(A,B,C,\mathbf{G})$ if it is clear from the context which transfer function $\mathbf{G}$ is to be used.
    \item
    One can straightforwardly extend \Cref{prop:extendedLinSystemBIBO} to the case of the extended linear system
    \begin{equation*}
        \Sigma\left(A, \begin{bmatrix} B & \widetilde{B} \end{bmatrix} ,  \begin{bmatrix} C \\ \widetilde{C} \end{bmatrix} ,  \begin{bmatrix} \mathbf{G} & C (\cdot I - A)^{-1} \widetilde{B} \\ \widetilde{C} (\cdot I - A)^{-1} B & \widetilde{C} (\cdot I - A)^{-1} \widetilde{B} \end{bmatrix}  \right)
    \end{equation*}
    where $\widetilde{B} \in \mathcal{L}(U,X)$ and $\widetilde{C} \in \mathcal{L}(X,Y)$.
\item We note that the assumption that $\Sigma(A,I,C)$ is ${\rm L}^{\infty}$-BIBO stable excludes boundary observation if $A$ even generates a strongly continuous group. Indeed, under this assumption it is shown in \cite[Proposition~6.6]{SWZ_BIBOStability} that $ \L^{\infty}$-BIBO stability of $\Sigma(A,I,C)$ implies that $C$ must be a bounded operator.
\item The exponential stability assumed in the proposition cannot be dropped, as the subsystem $\Sigma(A,I,I)$ is obviously not BIBO if e.g.\ $A=0$. 
\end{enumerate}
\end{remark}

\section{Different types of uniformly Lipschitz nonlinearities}\label{Section_NonlinearOperator}

In this section we consider global Lipschitz nonlinearities $f$ defined on interpolation spaces between $D(A)$ and $X$, if $A$ generates an analytic semigroup, and otherwise on $X$ itself.

The standing assumption of this section is that the (unbounded) linear operator $A: D(A)\subset X\to X$ is the infinitesimal generator of an exponentially stable $C_0$-semigroup $(T(t))_{t\geq 0}$, i.e. there exist $M\geq 1$ and $\omega>0$ such that for all $ t \geq 0$ it holds that
\begin{equation}\label{eq:exp_stable_sg}
\Vert T(t)\Vert\leq M \e^{-\omega t}.
\end{equation}

Additionally, if $A$ generates an analytic semigroup, then according to \cite[Chapter~6]{Pazy} the fractional powers of $-A$, $(-A)^\alpha, 0\leq \alpha\leq 1$, are well-defined and the operator $(-A)^\alpha$ is closed, linear and invertible. Moreover, $D((-A)^\alpha)$ is dense in $X$. As an additional property, the analyticity entails that
\begin{equation}\label{EstimateAnalytic}
    \Vert (-A)^\alpha T(t)\Vert\leq M_\alpha t^{-\alpha} \e^{-\delta t}
 \end{equation}
holds for all $t>0$, where the positive constant $\delta$ is such that the operator $A+\delta I$ is still the infinitesimal generator of an analytic and exponentially stable semigroup and $M_\alpha > 0$ depends only on $\alpha$, see e.g.~\cite[Chapter~2,~Theorem~6.13]{Pazy}. Note that \eqref{EstimateAnalytic} still holds with $\alpha=0$ if $A$ does not generate an analytic semigroup, where we denote $(-A)^0=I$. Indeed, we can choose $\delta= \omega$ and $M_\alpha=M$ with constants $M$ and $\omega$ from \eqref{eq:exp_stable_sg}.

If $A$ generates an analytic semigroup, we denote by $X_\alpha$, $0\leq\alpha\leq 1$ the space $D((-A)^\alpha)$ equipped with the norm $\Vert x\Vert_\alpha \coloneqq \Vert (-A)^\alpha x\Vert_X$, and otherwise we set $\alpha=0$ and write $(-A)^0=I$, $X_0=X$ and $\lVert \cdot \rVert_0 := \lVert \cdot \rVert_X$.

\begin{remark}\label{rem:(-A)^alpha_Linfty-adm}
    Assume that $A$ generates an analytic semigroup 
    \begin{enumerate}
        \item  Every operator $\Tilde{B}\in \mathcal{L}(U,X_{-\beta})$ with $0 \leq \beta <1$ is $\L^p$-admissible for $p > \frac{1}{1-\beta}$, see e.g.~\cite[Propositon~2.13]{SchwenningerISS}. 
        \item It follows from 1. that for $\alpha \in [0,1)$ the extension of $(-A)^\alpha$ to an operator in $\mathcal{L}(X,X_{- \alpha})$ is $\L^\infty$-admissible. Using \eqref{EstimateAnalytic}, we can give an upper bound of the admissibility constant. It holds for $\tilde{u} \in \L^\infty(\mathbb{R}_{\geq 0};X)$ that
        \begin{align}\label{eq:(-A)^alpha_Linfty_admissibility_constant}
            \begin{split}
                \left \| \int_0^t T(t-s) (-A)^{\alpha} \tilde{u}(s) \d s \right\|_X 
                &= \left \| \int_0^t (-A)^{\alpha} T(t-s) \tilde{u}(s) \d s \right\|_X \\
                &\leq M_\alpha \| \tilde{u}\|_{\L^\infty([0,t];X)} \int_0^t (t-s)^{-\alpha} \e^{-\delta(t-s)} \d s\\
                &\leq \frac{M_\alpha \Gamma(1-\alpha)}{\delta^{1-\alpha}}\|\tilde{u} \|_{\L^\infty([0,t];X)}.
            \end{split}
        \end{align}
        Estimate \eqref{eq:(-A)^alpha_Linfty_admissibility_constant} is still valid for $\alpha=0$ without the analyticity of the semigroup generated by $A$.\\
        Assuming just bounded analyticity of the semigroup, i.e. the semigroup is bounded on some sector, instead of exponential stability, \eqref{EstimateAnalytic} holds with $\delta=0$ and thus, we obtain
        \begin{equation*}
            \left \lVert \int_0^t T(t-s) (-A)^\alpha \tilde{u}(s) \d s \right\rVert_X \leq M \lVert \tilde{u} \rVert_{\L^\infty([0,t];X)} \int_0^t (t-s)^{-\alpha} \d s
        \end{equation*}
        which is also valid for not necessarily analytic semigroups if $\alpha=0$. Note that exponential stability together with analyticity implies bounded analyticity.
    \end{enumerate}
    
\end{remark}

Before discussing $\L^\infty$-BIBO stability of \eqref{eq:semilinStateSpaceNode} we give a result on the existence of unique mild solutions of this system (without considering an output). The used methods are well-known and similar results for slightly different situations are available (see e.g.~\cite[Section~6.3]{Pazy} or more recently \cite{Mironchenko_WellPosedness}). For the sake of completeness, we nevertheless give the details.

\begin{lemma}\label{lem:existence_of_mild_solution}
   Let $A$ be the generator of a bounded $C_0$-semigroup. If the semigroup is bounded analytic, let $\alpha \in [0,1)$; else, set $\alpha=0$. Moreover, let $B \in \mathcal{L}(U,X_{-(1-\alpha)})$ be such that $(-A)^\alpha B$ is $\L^\infty$-admissible and $f:\mathbb{R}_{\geq 0} \times X_\alpha \rightarrow X$ is locally Lipschitz in the following sense: there exists a measurable function $g: \mathbb{R}_{\geq 0}^2 \rightarrow \mathbb{R}_{\geq 0}$ such that the following properties hold.
   \begin{itemize}
       \item  $g(\cdot,0) \in \L^\infty_{\rm loc}(\mathbb{R}_{\geq 0};\mathbb{R}_{\geq 0})$,
       \item $g(s,s)=0$ for all $s \geq 0$,
       \item for every bounded set $V \subseteq \mathbb{R}_{\geq 0} \times X_\alpha$ there exists a constant $L>0$ and $0<\nu \leq 1$ such that for every $(t_1,x_1) , (t_2,x_2) \in V$ it holds that
    \begin{equation}\label{eq:assump_time_varying_f}
      \lVert f(t_1,x_1) - f(t_2,x_2) \rVert_X \leq L ( g(t_1,t_2) + \lVert x_1 - x_2 \rVert_\alpha).
    \end{equation}
   \end{itemize}  
    Then for every $t_0 \geq 0$, $x_0 \in X_\alpha$ and $u \in \L^\infty([t_0,\infty);U)$, the system
    \begin{equation}\label{eq:time_varying_system}
        \left\{
        \begin{aligned}
          \dot{x}(t) &= Ax(t) + Bu(t) + f(t,x(t)), \quad t > t_0, \\
             x(t_0)&=x_0
         \end{aligned}
         \right.
    \end{equation}
    admits a unique mild solution $x\in \L^\infty([t_0,t_1];X_\alpha)$ for some $t_1>t_0$ ,
    i.e. $x$ satisfies the implicit equation
    \begin{equation*}
        x(t) = T(t-t_0)x_0 + \int_{t_0}^t T(t-s)Bu(s) \d s + \int_{t_0}^t T(t-s)f(s,x(s)) \d s, \qquad t \in [t_0,t_1].
    \end{equation*}
    Furthermore, if $t_{\rm max}>t_0$ denotes the supremum over all $t_1$ such that \eqref{eq:time_varying_system} admits a solution on $[t_0,t_1]$ in the above sense, then we have the finite blow-up property, i.e.\
    \begin{equation*}
        t_{\rm max}< \infty \quad\implies\quad \lim_{t \nearrow t_{\rm max}} \lVert x(t) \rVert_\alpha = \infty.
    \end{equation*}
        Additionally, if there exists a nondecreasing function $k \in \mathcal{C}([t_0,\infty);\mathbb{R})$ such that for every $t \geq t_0$ and $x \in X_\alpha$,
    \begin{equation}\label{eq:existence_of_solutions_time_bound_for_globalness}
       \lVert f(t,x) \rVert_X \leq k(t) (1 + \lVert x \rVert_\alpha),
    \end{equation}
    then the solution $x$ lies in $\L^\infty_{\mathrm loc}([t_0,\infty);X_\alpha)$.
\end{lemma}

\begin{proof}
    First, we prove that for every $t_0 \geq 0$, $x_0 \in X_\alpha$ and $u \in \L^\infty(\mathbb{R}_{\geq 0};U)$ there exists $t_1>t_0$ such that \eqref{eq:time_varying_system} admits a solution $x \in \L^\infty([t_0,t_1];X_\alpha)$. Moreover we prove that $t_1=t_0 + \delta$ can be chosen such that $\delta>0$ does not depend on the initial time $t_0$ and a set of bounded initial values $x_0$ in $X_\alpha$. By $C_{1,t}$ and $C_{2,t}$ we denote the $\L^\infty$-admissibility constants of $(-A)^\alpha B$ and $(-A)^\alpha$, respectively. Inequality \eqref{eq:(-A)^alpha_Linfty_admissibility_constant} yields that $C_{2,t} \to 0$ as $t \to 0^{+}$.\\
    We denote by $M\geq 1$ the constant such that $\lVert T(t) \rVert \leq M$ for every $t \geq 0$. Let $t_1'>t_0$. For $r>0$, $x_0 \in X_\alpha$ with $\lVert x_0 \rVert_\alpha \leq r$ and $u \in \L^\infty(\mathbb{R}_{\geq 0};U)$ we choose
    \begin{equation*}
        m \coloneqq 2Mr+C_{1,\infty} \lVert u \rVert_{\L^\infty([0,t_1'];U)} >0.
    \end{equation*}
    Further, let $L>0$ be a constant satisfying \eqref{eq:assump_time_varying_f} for $V=( [t_0,t_1'] \times \{ x \in X_\alpha \mid \lVert x \rVert_\alpha \leq m\}) \cup \{(0,0)\}$. Choose $0<\delta <t_1'$ such that
    \begin{align*}
        C_{2,\delta} \leq \min \left\{ \frac{Mr}{ L ( \lVert g(\cdot,0) \rVert_{\L^\infty([0,t_1'];\mathbb{R}_{\geq 0})}  + m ) +\lVert f(0,0) \rVert_X }, \frac{1}{2L} \right\}.
    \end{align*}
    Note that $\delta$ depends on $r,(T(t))_{t \geq 0},\alpha,u,f$ and $t_1'>t_0$ but not on $t_0>0$ and $x_0$ with $\lVert x_0 \rVert_\alpha \leq r$.\\
    For $y \in S\coloneqq \{ y \in \L^\infty([t_0,t_0+\delta];X) \mid \lVert y \rVert_{\L^\infty([t_0,t_0+\delta];X)} \leq m \}$ and $t \in [t_0,t_0+\delta]$ we define
    \begin{equation*}
        (Fy)(t) \coloneqq T(t-t_0) (-A)^\alpha x_0 + \int_{t_0}^{t} T(t-s)(-A)^\alpha B u(s) \d s + \int_{t_0}^{t} T(t-s)(-A)^\alpha f(s, (-A)^{-\alpha} y(s) ) \d s.
    \end{equation*}
    We will prove that $F$ is a contraction on $S$. First note that $F$ is well-defined since
    \begin{align}\label{eq:bound_of_mild_solution}
        \begin{split}
            \lVert (Fy)(t) \rVert_X 
             &\leq M \lVert x_0 \rVert_\alpha + C_{1,\delta} \lVert u \rVert_{\L^\infty([t_0,t];U)} + C_{2,\delta} \lVert f(\cdot, (-A)^{-\alpha} y(\cdot)) \rVert_{\L^\infty([t_0,t];X)}\\
             &\leq M r + C_{1,\infty} \lVert u \rVert_{\L^\infty(\mathbb{R}_{\geq 0};U)} \\
             &+ C_{2,\delta}\left( L ( \lVert g(\cdot,0)\rVert_{\L^\infty([t_0,t];\mathbb{R}_{\geq 0})}  + \lVert (-A)^{-\alpha} y\rVert_{\L^\infty([t_0,t];X_\alpha)} ) + \lVert f(0,0) \rVert_X \right)\\
             &\leq M r + C_{1,\infty} \lVert u \rVert_{\L^\infty(\mathbb{R}_{\geq 0};U)} \\
             &+  C_{2,\delta}\left( L ( \lVert g(\cdot,0)\rVert_{\L^\infty([0,t_1'];\mathbb{R}_{\geq 0})}  + m ) +\lVert f(0,0) \rVert_X \right)\\
             &\leq m,
     \end{split}
    \end{align}
    where we used \eqref{eq:assump_time_varying_f} in the second last step. The contractivity of $F$ follows from
    \begin{align*}
        \lVert (Fy_1)(t) - (Fy_2)(t) \rVert_X
        &= \left\lVert \int_{t_0}^t T(t-s) (-A)^\alpha [ f(s,(-A)^{-\alpha} y_1(s)) - f(s,(-A)^{-\alpha} y_2(s))] \d s \right\rVert_X\\
        &\leq C_{2,\delta} L \lVert y_1 - y_2 \rVert_{\L^\infty([t_0,t];X)}
    \end{align*}
    for every $y_1,y_2 \in S$. Thus, there exists a unique $y \in S$ such that
    \begin{equation*}
        y(t) = T(t-t_0) (-A)^\alpha x_0 + \int_{t_0}^t T(t-s) (-A)^\alpha Bu(s) \d s + \int_{t_0}^t T(t-s) (-A)^\alpha f(s,(-A)^{-\alpha} y(s)) \d s
    \end{equation*}
    for almost every $t \in [t_0, t_0 + \delta]$. We obtain from \eqref{eq:assump_time_varying_f} that $f(\cdot,(-A)^{-\alpha}y(\cdot)) \in \L^\infty([t_0,t_0+\delta];X)$ and hence, the linear system
    \begin{equation}\label{eq:proof_x_mild_solution}
        \left\{
        \begin{aligned}
            \dot{x}(t) &= Ax(t) + Bu(t) + f(t,(-A)^{-\alpha}y(t)), \quad t>t_0\\
            x(t_0)&=x_0
        \end{aligned}
        \right.
    \end{equation}
    admits a unique mild solution $x  \in \L^\infty([t_0,t_0+\delta];X)$ given by
    \begin{equation*}
       x(t) = T(t)x_0 + \int_{t_0}^t T(t-s) Bu(s) \d s + \int_{t_0}^t T(t-s) f(s,(-A)^{-\alpha}y(s)) \d s,
    \end{equation*}
     where each term on the right hand-side lies in $D((-A)^\alpha)$ for almost every $t \in [t_0,t_0+\delta]$ by analyticity of the semigroup.\\
    It follows that $ (-A)^\alpha x(t) = (Fy)(t) =y(t)$  for almost every $t \in [t_0,t_0 + \delta]$ and thus, $x \in \L^\infty([t_0,t_0+\delta];X_\alpha)$.\\
    For given $t_0 \geq 0$, $x_0 \in X_\alpha$ and $u \in \L^\infty(\mathbb{R}_{\geq 0};U)$ we denote by $t_{\rm max}$ the supremum over all $t_1>0$ such that \eqref{eq:time_varying_system} admits a solution $x \in \L^\infty([t_0,t_1];X_\alpha)$. If $t_{\rm max}< \infty$ and $\lim_{t \nearrow t_{\rm max}} \lVert x(t) \rVert_\alpha < \infty$, then it follows that
    \begin{equation*}
        r \coloneqq \lVert x \rVert_{\L^\infty([t_0,t_{\rm max}];X_\alpha)} < \infty.
    \end{equation*}
    Let $(t_n)_{n \in \mathbb{N}}$ be an increasing sequence in $[t_0,t_{\rm max})$ converging to $t_{\rm max}$. From the previous argumentation we can find $\delta>0$ independent of $n \in \mathbb{N}$ such that the system
    \begin{equation*}
        \left\{
        \begin{aligned}
            \dot{x}_n(t) &= Ax_n(t) + Bu(t) + f(t,x_n(t)), \quad t>t_n, \\
            x_n(t_n) &= x(t_n)
        \end{aligned}
        \right.
    \end{equation*}
    admits unique solutions $x_n \in \L^\infty([t_n,t_n + \delta];X_\alpha)$. But then, we can extend the solution $x$ by $x_n$ to a solution on $[t_0,t_n + \delta)$. Since $t_n + \delta > t_{\rm max}$ for large $n$, this contradicts the maximality of $t_{\rm max}$.\\
    The last statement follows as in \cite[Theorem~6.3.3]{Pazy}. \qed
\end{proof}

\begin{remark}
    In the situation of \Cref{lem:existence_of_mild_solution}, if $A$ generates an analytic semigroup and $\alpha \in (0,1)$, then the solution $x$ lies also in $\mathcal{C}([t_0,t_1];X)$ or $\mathcal{C}([t_0,\infty);X)$, respectively, since $x$ is the mild solution of the linear system \eqref{eq:proof_x_mild_solution} with $\L^p$-admissible operators $B$ and $I$ for some $p \in [1,\infty)$ (c.f. \Cref{rem:(-A)^alpha_Linfty-adm}) as can be seen from \cite[Proposition~2.3]{Weiss89ControlOperators}.
\end{remark}

\begin{remark}\label{rem:lin_syst_alpha_BIBO}
    \begin{enumerate}
        \item If $A$ generates an exponentially stable and analytic semigroup, let $\alpha \in [0,1)$; else, set $\alpha=0$. If $B \in \mathcal{L}(U,X_{-(1-\alpha)})$ such that $(-A)^\alpha B$ is $\L^\infty$-admissible, then the mild solution of $\Sigma(A,[\begin{smallmatrix} B & I\end{smallmatrix}],[\begin{smallmatrix}C\\ I\end{smallmatrix}])$ for $x_0 \in X_\alpha$, $u \in \L^\infty(\mathbb{R}_{\geq 0};U)$, $\tilde{u} \in \L^\infty(\mathbb{R}_{\geq 0};X)$ satisfies
        \begin{align*}
            \lVert x(t) \rVert_\alpha 
            &\leq  M \e^{-\omega t } \lVert x_0 \rVert_\alpha + C_{1,\infty} \lVert u \rVert_{\L^\infty([0,t];U)} + C_{2,\infty} \lVert \tilde{u} \rVert_{\L^\infty([0,t];X)},
        \end{align*}
        where $C_{i,\infty}$, $i=1,2$, are the infinite-time $\L^\infty$-admissibility constants of $(-A)^\alpha B$ and $(-A)^\alpha$, respectively, and constants $M,\omega$ from \eqref{eq:exp_stable_sg}. From \Cref{rem:(-A)^alpha_Linfty-adm}, we know that $C_{2,\infty} \leq \frac{M_\alpha \Gamma(1-\alpha)}{\delta^{1-\alpha}}$ with constants from \eqref{EstimateAnalytic}.
     \item From the considerations in 1. and the fact that the transfer function of $\Sigma(A,[\begin{smallmatrix} B & I\end{smallmatrix}],[\begin{smallmatrix}C\\ I\end{smallmatrix}])$ is not only mapping into $\mathcal{L}(U,Y \times X)$ but also into $\mathcal{L}(U,Y \times X_\alpha)$, we obtain, that $\Sigma(A,[\begin{smallmatrix} B & I\end{smallmatrix}],[\begin{smallmatrix}C\\ I\end{smallmatrix}])$ is $\L^\infty$-BIBO stable with respect to the spaces $(U \times X, X ,Y \times X)$ if and only if it is $\L^\infty$-BIBO stable with respect to the spaces $(U \times X, X_\alpha, Y \times X_\alpha)$. Hence, if one of the above system nodes is $\L^\infty$-BIBO stable, by linearity there exist constants $K_1,K_2>0$ such that for $x_0=0$ and every $u \in \L^\infty(\mathbb{R}_{\geq 0};U)$, $\tilde{u} \in \L^\infty(\mathbb{R}_{\geq 0};X)$ the output $\tilde{y}$ satisfies
    \begin{equation}\label{eq:lin_syst_alpha_BIBO_estimate}
        \lVert \tilde{y} \rVert_{\L^\infty([0,t];Y \times X_\alpha)} \leq K_1 \lVert u \rVert_{\L^\infty([0,t];U)} + K_2 \lVert \tilde{u} \rVert_{\L^\infty([0,t];X)}.
     \end{equation}
     \end{enumerate}
\end{remark}

Next we present our main theorem on $\L^\infty$-BIBO stability of the semilinear state space system \eqref{eq:semilinStateSpaceNode} for globally Lipschitz continuous functions $f:X_\alpha \rightarrow X$, i.e. there exists a constant $L>0$ such that for every $x_1,x_2\in X_\alpha$ it holds that
 \begin{equation}
    \label{eq:GlobalLipschitz}
    \lVert f(x_1) - f(x_2) \rVert_X \leq L \lVert x_1 - x_2 \rVert_\alpha.
\end{equation}

\begin{theorem}\label{thm:BIBO_under_global_Lip_on_X}
    Let $A$ generate an exponentially stable $C_0$-semigroup. If the semigroup is analytic, let $\alpha \in [0,1)$; else, set $\alpha=0$. Let $B \in \mathcal{L}(U,X_{-(1-\alpha)})$ be such that $(-A)^\alpha B$ is $\L^\infty$-admissible, $f$ satisfy \eqref{eq:GlobalLipschitz} with constant $L>0$ and $\Sigma(A,[\begin{smallmatrix} B & I\end{smallmatrix}],[\begin{smallmatrix}C\\ I\end{smallmatrix}])$ be $\L^\infty$-BIBO stable. If $C_{2,\infty}L<1$, where $C_{2,\infty}$ is the infinite-time $\L^\infty$-admissibility constant of $(-A)^\alpha$, then the output $y$ of \eqref{eq:semilinStateSpaceNode} with initial value $x_0 =0$ and input $u \in \L^\infty(\mathbb{R}_{\geq 0};U)$ satisfies the following inequality for some $K,\mathfrak{K} \geq 0$ and every $t \geq 0$
    \begin{equation}\label{IneqBIBO_Glob_Lipschitz}
        \lVert y \rVert_{\L^\infty([0,t];Y)} \leq K \Vert u\Vert_{\L^\infty([0,t];U)} + \mathfrak{K},
    \end{equation}
    This particularly means that the semilinear state space system \eqref{eq:semilinStateSpaceNode} is $\L^\infty$-BIBO stable.
\end{theorem}

\begin{proof}
    First note that for any $u \in \L^\infty(\mathbb{R}_{\geq 0};U)$ and any initial value in $X$ the mild solution $x \in \mathcal{C}(\mathbb{R}_{\geq 0};X)$ of \eqref{eq:time_varying_system} exists due to \Cref{lem:existence_of_mild_solution} by our assumptions on $B$ and $f$. Further, note that $x$ is also the state trajectory of the linear system node $\Sigma(A,[\begin{smallmatrix}B & I\end{smallmatrix}],[\begin{smallmatrix}C\\ I\end{smallmatrix}])$ with input  $[\begin{smallmatrix} u \\ f(x(\cdot)) \end{smallmatrix}] \in \L^\infty_{\mathrm loc}(\mathbb{R}_{\geq 0};U \times X)$ and that the corresponding output is given by $\tilde{y} = [\begin{smallmatrix} y \\ x \end{smallmatrix}]$, where $y$ is given by \eqref{eq:outputDistExpression}. Hence, by \Cref{def:mildSolution}, $(u,x,y)$ is a mild solution of the semilinear system $(\Sigma,f)$. Furthermore, because the linear system node  $\Sigma(A,[\begin{smallmatrix}B & I\end{smallmatrix}],[\begin{smallmatrix}C\\ I\end{smallmatrix}])$ is $\L^\infty$-BIBO stable, we have $\tilde{y} \in \L^\infty(\mathbb{R}_{\geq 0};Y \times X_\alpha)$ and thus $y \in \L^\infty(\mathbb{R}_{\geq 0};Y)$.\\
    
    Let $x_0=0$ and $u \in \L^\infty(\mathbb{R}_{\geq 0};U)$. We deduce from \Cref{rem:lin_syst_alpha_BIBO} and \eqref{eq:GlobalLipschitz} that
    \begin{equation*}
    \lVert x \rVert_{\L^\infty([0,t];X_\alpha)} \leq  C_{1,\infty} \lVert u \rVert_{\L^\infty([0,t];U)} + C_{2,\infty} L \lVert x \rVert_{\L^\infty([0,t];X_\alpha)} + C_{2,\infty} \lVert f(0) \rVert_X
    \end{equation*}
    and thus, since $C_{2,\infty}L<1$,
    \begin{equation*}
     \lVert x \rVert_{\L^\infty([0,t];X_\alpha)} \leq   \frac{C_{1,\infty}}{1-C_{2,\infty} L} \lVert u \rVert_{\L^\infty([0,t];U)} + \frac{C_{2,\infty}}{1-L C_{2,\infty}} \lVert f(0) \rVert_X.
    \end{equation*}
    Combining this with \eqref{eq:lin_syst_alpha_BIBO_estimate} for $\tilde{u}=f(x)$ and applying \eqref{eq:GlobalLipschitz} once more yield that
    \begin{equation*}
    \lVert y \rVert_{\L^\infty([0,t];Y)} \leq  \lVert \tilde{y} \rVert_{\L^\infty([0,t];Y \times X_\alpha)} \leq  \left(K_1 + \frac{K_2 L C_{1,\infty}}{1-C_{2,\infty} L}\right) \lVert u \rVert_{\L^\infty([0,t];U)} + \mathfrak{K}  
    \end{equation*}
    with $\mathfrak{K}=\left( K_2 + \frac{K_2 C_{2,\infty} L}{1-C_{2,\infty} L} \right) \lVert f(0) \rVert_X$. \qed
\end{proof}

\begin{cor}\label{cor:BIBO_under_global_Lip_on_X}
    Let the assumptions of \Cref{thm:BIBO_under_global_Lip_on_X} hold. By $M_\alpha$, $\delta$ and $K_2$ we denote the constants from \eqref{EstimateAnalytic} and \eqref{eq:lin_syst_alpha_BIBO_estimate}, respectively. If either, $\frac{LM_\alpha \Gamma(1-\alpha)}{\delta^{1-\alpha}}<1$ or $K_2 L<1$, then \eqref{IneqBIBO_Glob_Lipschitz} holds and hence, \eqref{eq:semilinStateSpaceNode} is $\L^\infty$-BIBO.
\end{cor}

\begin{proof}
    By the definition of the admissibility constant, $C_{2,\infty}$ is the smallest, time independent constant such that the mild solution $x$ of $\Sigma(A,[\begin{smallmatrix}B & I\end{smallmatrix}],[\begin{smallmatrix}C\\ I\end{smallmatrix}])$ for $x_0=0$, $u=0$ and $\tilde{u} \in \L^\infty(\mathbb{R}_{\geq 0};X)$ satisfies for every $t \geq 0$ that
    \begin{equation*}
        \lVert x(t) \rVert_\alpha \leq C_{2,\infty} \lVert \tilde{u} \rVert_{\L^\infty([0,t];X)}.
    \end{equation*}
    Hence, it holds that $C_{2,\infty} \leq K_2$ and by \Cref{rem:(-A)^alpha_Linfty-adm} also $C_{2,\infty} \leq \frac{M_\alpha \Gamma(1-\alpha)}{\delta^{1-\alpha}}$. The assertion is now a consequence of \Cref{thm:BIBO_under_global_Lip_on_X}.\qed 
\end{proof}

\begin{remark}
    In the situation of \Cref{cor:BIBO_under_global_Lip_on_X}, it is possible to improve the constants in \eqref{IneqBIBO_Glob_Lipschitz} by replacing $C_{2,\infty}$ by $K_2$ or $\frac{M_\alpha \Gamma(1-\alpha)}{\delta^{1-\alpha}}$ suitably in the proof of \Cref{thm:BIBO_under_global_Lip_on_X}.
\end{remark}

\begin{remark}
    \Cref{thm:BIBO_under_global_Lip_on_X} and \Cref{cor:BIBO_under_global_Lip_on_X} can be easily generalized to nonlinearities $f$ depending also on time $t\geq 0$ and satisfying \eqref{eq:assump_time_varying_f} and \eqref{eq:existence_of_solutions_time_bound_for_globalness} for a bounded function $k\in \mathcal{C}(\mathbb{R}_{\geq 0};\mathbb{R}_{\geq 0})$. Indeed, on has to replace the Lipschitz constant $L$ in the smallness conditions by $\lVert k \Vert_{\L^\infty(\mathbb{R}_{\geq 0};\mathbb{R}_{\geq 0})}$ in the statement and the proof.
\end{remark}

\section{Locally Lipschitz nonlinearities with internal control}\label{Section_Local_Lipschitz}
We consider the following heat equation with Neumann boundary conditions, internal friction and which is subject to internal control and cubic nonlinearity on an open and bounded domain $\Omega \subset \mathbb{R}^d$, $d \leq 3$, with Lipschitz boundary $\partial \Omega$
\begin{equation}
    \left\{
    \begin{aligned}
         &\frac{\partial x}{\partial t}(\zeta,t) = \Delta x(\zeta,t) - x^3(\zeta,t) + Bu(t), && \zeta \in \Omega, t \in \mathbb{R}_{\geq 0},\\
         &x(\zeta,0) = x_0(\zeta), && \zeta \in \Omega,\\
         &\frac{\partial x}{\partial \nu}(\zeta,t) =0, && \zeta \in \partial \Omega, t \in \mathbb{R}_{\geq 0},
    \end{aligned}
    \right.
    \label{HeatNeumannInternalControl}
\end{equation}
 where $\nu$ is the outward pointing unit normal vector at the boundary and $\frac{\partial}{\partial \nu} \in \mathcal{L}(\H^1(\Omega;\mathbb{R}),\L^2(\partial \Omega;\mathbb{R}))$ is the Neumann trace operator, which coincides with the normal derivative on smooth functions.\\
In an abstract formulation, \eqref{HeatNeumannInternalControl} may be written as $\dot{x}(t) = Ax(t) + f(x(t)) + Bu(t)$, $x(0) = x_0$, where
\begin{itemize}
    \item the state space $X$ is the space of square integrable functions on $\Omega$, i.e. $X = \L^2(\Omega;\mathbb{R})$;
    \item $A: D(A)\subset X\to X$, given by 
    \begin{equation*}
        Ax = \Delta x , \quad \text{for} \quad  x\in D(A) = \left\{x\in \H^2(\Omega;\mathbb{R}) \colon \frac{\partial x}{\partial \nu} =0  \text{ on } \partial \Omega \right\},
    \end{equation*}
    is a Riesz-spectral operator that generates a bounded analytic $C_0$-semigroup. Moreover, $-A$ is coercive and self-adjoint;
    \item $f: X_\frac{1}{2} \to X$ is defined by $f(x) = -x^3$;
    \item the input function $u$  is assumed to take values in an arbitrary Banach space $U$ and the control operator $B:U\to X$ is such that $B\in\mathcal{L}(U,X)$.
\end{itemize}

Observe that, thanks to this setting, the Hilbert space $X_\frac{1}{2}$ is given by $X_\frac{1}{2} = \H^1(\Omega;\mathbb{R})$ with the norm $\Vert x\Vert_\frac{1}{2}^2 = \Vert x\Vert_{\L^2(\Omega;\mathbb{R})}^2 + \Vert \nabla x \Vert_{\L^2(\Omega;\mathbb{R})}^2$, which shows that $f$ is well-defined (even continuous), since $\H^1(\Omega;\mathbb{R})$ is continuously embedded in $\L^6(\Omega;\mathbb{R})$. The following theorem gives an upper bound on the $X_\frac{1}{2}$-norm of the state trajectory of \eqref{HeatNeumannInternalControl}.
The following result establishes well-posedness of the pde \eqref{HeatNeumannInternalControl}. Similar results are well-known; e.g.\ for $u=0$ see \cite[Chapter 11]{CurtainZwartNew}. Our focus lies in an estimate of the solution in the $X_{\frac{1}{2}}$-norm.
\begin{theorem}\label{thm:heat_ISS}
    Let $X = \L^2(\Omega;\mathbb{R})$ and $B\in\mathcal{L}(U,X)$. For any initial condition $x_0\in X_\frac{1}{2}$ and  control input $u\in \L^\infty(\mathbb{R}_{\geq 0};U)$, the heat equation \eqref{HeatNeumannInternalControl} admits a unique mild solution $x \in \H^1_{\rm loc}(\mathbb{R}_{\geq 0};X) \cap \mathcal{C}(\mathbb{R}_{\geq 0};X_\frac{1}{2}) \cap \L^2_{\rm loc}(\mathbb{R}_{ \geq 0};X_1)$ which satisfies the estimate 
    \begin{align*}
        &\Vert (-A)^\frac{1}{2}x(t)\Vert_X^2\leq \frac{1}{2}\left[\Vert x_0\Vert_X^2 + 2\Vert (-A)^\frac{1}{2}x_0\Vert_X^2 + \int_\Omega x_0^4(\zeta)d\zeta\right]\e^{-\rho t} + \frac{\lambda}{2}\int_0^t \e^{-\rho(t-s)}\Vert u(s)\Vert_U^2 \d s,
    \end{align*}
    for all $t\geq 0$ and some $\lambda, \rho >0$ independent of $t$.
\end{theorem}

\begin{proof}
    Let $x_0 \in X_{\frac{1}{2}}$ and $u \in \L^\infty(\mathbb{R}_{\geq 0};U)$. Since $f:X_{\frac{1}{2}} \rightarrow X$ is locally Lipschitz we deduce from \Cref{lem:existence_of_mild_solution} the existence of a unique mild solution $x \in \mathcal{C}([0,t_1];X) \cap \L^\infty([0,t_1];X_{\frac{1}{2}})$ for some $t_1>0$. Consequently, $\tilde{u}\coloneqq f(x(\cdot)) \in \L^\infty([0,t_1];X_{\frac{1}{2}}) \subseteq \L^2([0,t_1];X_{\frac{1}{2}})$. Since $x$ is also the mild solution of the linear system
    \begin{equation*}
        \left\{
        \begin{aligned}
            \dot{x}(t) &= Ax(t) + Bu(t) + \tilde{u}(t),\\
            x(0) &= x_0,
        \end{aligned}
        \right.
    \end{equation*}
    where the control operators $B$ and $I$ are bounded as operators into $X$ and therefore also into $X_{\frac{1}{2}}$. The maximal regularity property of the analytic semigroup and \cite[Proposition~6.5]{TucsnakWeiss2014} yield that $x \in \H^1([0,t_1];X_{-\frac{1}{2}}) \cap \mathcal{C}([0,t_1];X) \cap L^2([0,t_1];X_{\frac{1}{2}})$ and
    \begin{align*}
        \lVert x(t) \rVert_X^2 + 2 \int_0^t \lVert x(s) \rVert_{\frac{1}{2}}^2 \d s =\lVert x_0 \rVert_X^2 + 2 \int_0^t \langle  x(s), Bu(s) \rangle_X \d s + 2 \int_0^t \langle x(s), f(x(s)) \rangle_X \d s
    \end{align*}
    for every $t \in[0,t_1]$.
    Similarly, since $z=(-A)^{\frac{1}{2}}x$ is the mild solution of the linear system
    \begin{equation*}
        \left\{
        \begin{aligned}
            \dot{z}(t) &= Az(t) + (-A)^{\frac{1}{2}} Bu(t) + (-A)^{\frac{1}{2}} \tilde{u}(t),\\
            z(0) &= (-A)^{\frac{1}{2}} x_0, 
        \end{aligned}
       \right.
    \end{equation*}
    we obtain $z \in \H^1([0,t_1];X_{-\frac{1}{2}}) \cap \mathcal{C}([0,t_1];X) \cap \L^2([0,t_1];X_{\frac{1}{2}})$, which translates to $x \in \H^1([0,t_1];X) \cap \mathcal{C}([0,t_1];X_{\frac{1}{2}}) \cap \L^2([0,t_1];X_{1})$. Similar as before, it holds that
    \begin{align*}
        \lVert x(t) \rVert_{\frac{1}{2}}^2 + 2 \int_0^t \lVert x(s) \rVert_1^2 \d s = \lVert x_0 \rVert_{\frac{1}{2}}^2 - 2 \int_0^t \langle Ax(s), Bu(s) \rangle_X \d s - 2 \int_0^t \langle Ax(s) , f(x(s))\rangle_X \d s
    \end{align*}
    for every $t \in [0,t_1]$. In order to get an ISS estimate for \Cref{HeatNeumannInternalControl} on the space $X_\frac{1}{2}$, let us consider the following Lyapunov functional candidate
\begin{equation*}
    V(x(t)) \coloneqq \Vert x(t)\Vert_X^2 + 2\Vert (-A)^\frac{1}{2}x(t)\Vert_X^2 + \int_\Omega x^4(\zeta,t) \d \zeta,
\end{equation*}
which is almost everywhere differentiable on $(0,t_1)$ by the regularity of the solution $x$ and the representations of the squared $X$-norm and $X_{\frac{1}{2}}$-norm of $x(t)$. We obtain for the derivative
\begin{align*}
    \dot{V}(x(t)) =& -2 \lVert x(t) \rVert_{\frac{1}{2}}^2 + 2 \langle x(t), Bu(t) \rangle_X + 2 \langle x(t), f(x(t)) \rangle_X - 4 \lVert x(t) \rVert_{1}^2 - 4 \langle Ax(t),Bu(t)\rangle_X\\
    & - 4 \langle Ax(t), f(x(t)) \rangle_X - 4\langle f(x(t)), Ax(t) + Bu(t) + f(x(t)) \rangle_X\\
    =& -2\lVert x(t) \rVert_{\frac{1}{2}}^2 +  2\langle x(t),Bu(t)\rangle_X + 2\langle x(t),f(x(t))\rangle_X\\
    &- 4\Vert Ax(t)+f(x(t))\Vert_X^2 - 4\langle Ax(t)+f(x(t)),Bu(t)\rangle_X.
\end{align*}
for every $t \in (0,t_1)$.
According to a generalization of Young's inequality and the boundedness of the operator $B$ from $U$ into $X$, we have
\begin{align*}
    \dot{V}(x(t))\leq& -2\Vert x(t)\Vert_{\frac{1}{2}}^2 - 2\int_\Omega x^4(\zeta,t) \d \zeta  + \epsilon\Vert x(t)\Vert_X^2 + \frac{1}{\epsilon}\Vert B\Vert_{\mathcal{L}(U,X)}^2\Vert u(t)\Vert_U^2 - 4\Vert Ax(t)+f(x(t))\Vert_X^2\\
    & +2\eta\Vert Ax(t)+f(x(t))\Vert_X^2 + \frac{2}{\eta}\Vert B\Vert_{\mathcal{L}(U,X)}\Vert u(t)\Vert_U^2,
\end{align*}
for arbitrary $\epsilon, \eta > 0$. In order to get a satisfactory ISS estimate in $X_\frac{1}{2}$-norm, one shall choose $\epsilon$ and $\eta$ such that $\epsilon<2$ and $\eta<2$. Moreover, by noting that $\Vert x(t)\Vert_X\leq \Vert x(t)\Vert_\frac{1}{2}$, the following inequalities are satisfied
\begin{align}
    \dot{V}(x(t))\leq& (-2+\epsilon)\Vert x(t)\Vert_\frac{1}{2}^2 - 2\int_\Omega x^4(\zeta,t) \d \zeta + (-4+2\eta)\Vert Ax(t)+f(x(t))\Vert_X^2\nonumber\\
    &+ \left(\frac{1}{\epsilon}+\frac{2}{\eta}\right)\Vert B\Vert_{\mathcal{L}(U,X)}\Vert u(t)\Vert_U^2\nonumber\\
    \leq& -\kappa\Vert (-A)^\frac{1}{2}x(t)\Vert_X^2 - 2\int_\Omega x^4(\zeta,t) \d \zeta + \lambda\Vert u(t)\Vert_U^2,\label{IneqV}
\end{align}
where $\kappa \coloneqq -2+\epsilon$ and $\lambda \coloneqq \left(\frac{1}{\epsilon}+\frac{2}{\eta}\right)\Vert B\Vert_{\mathcal{L}(U,X)}$. Besides, observe that 
\begin{equation}
    V(x(t))\geq \Vert (-A)^\frac{1}{2}x(t)\Vert_X^2 + \int_\Omega x^4(\zeta,t) \d \zeta =: W(x(t))
    \label{IneqLeftV}
\end{equation}
and that 
\begin{equation}
    V(x(t))\leq 3 W(x(t)).
    \label{IneqRightV}
\end{equation}
Combining \eqref{IneqV}, \eqref{IneqLeftV} and \eqref{IneqRightV} implies 
\begin{equation*}
    \dot{V}(x(t))\leq \frac{\max\{-\kappa,-2\}}{3}V(x(t)) + \lambda\Vert u(t)\Vert_U^2,
\end{equation*}
which, by using Gronwall's inequality, has the consequence that
\begin{equation*}
    V(x(t))\leq V(x(0)) \e^{-\rho t} + \lambda\int_0^t \e^{-\rho(t-s)}\Vert u(s)\Vert_U^2 \d s,
\end{equation*}
where $\rho \coloneqq -\frac{\max\{-\kappa,-2\}}{3}$. This entails that the $X_\frac{1}{2}$-norm of $x(t)$ can be upper-bounded as
\begin{align}
    &\Vert (-A)^\frac{1}{2}x(t)\Vert_X^2\leq \frac{1}{2}\left[\Vert x_0\Vert_X^2 + 2\Vert (-A)^\frac{1}{2}x_0\Vert_X^2 + \int_\Omega x_0^4(\zeta) \d \zeta\right]\e^{-\rho t} + \frac{\lambda}{2}\int_0^t \e^{-\rho(t-s)}\Vert u(s)\Vert_U^2 \d s,
    \label{ISS_Heat}
\end{align}
which completes the proof.\qed
\end{proof}

\begin{cor}
    The heat equation with output 
    \begin{equation*}
    \left\{
    \begin{aligned}
         \frac{\partial x}{\partial t}(\zeta,t) &= \Delta x(\zeta,t) - x^3(\zeta,t) + Bu(t), && \zeta \in \Omega, t \in \mathbb{R}_{\geq 0},\\
         x(\zeta,0) &= x_0(\zeta), && \zeta \in \Omega,\\
         \frac{\partial x}{\partial \nu}(\zeta,t) &=0, && \zeta \in \partial \Omega, t \in \mathbb{R}_{\geq 0},\\
         y(t) &= Cx(\cdot,t), && t \in \mathbb{R}_{\geq 0},     
    \end{aligned}
    \right.
    \end{equation*}
    with state space $X=\L^2(\Omega;\mathbb{R})$, input space $U$, control operator $B \in \mathcal{L}(U,X)$, output space $Y$ and ouput operator $C \in \mathcal{L}(X_\frac{1}{2},Y)$ is a $\L^\infty$-BIBO stable semilinear state space system as defined in \Cref{def:semilinear_state_space_system}.
\end{cor}

\begin{proof}
    \Cref{thm:heat_ISS} implies that \eqref{HeatNeumannInternalControl} admits for every $u \in L^\infty(\mathbb{R}_{\geq 0};U)$ and $x_0=0$ a mild solution $x \in \H^1_{\rm loc}(\mathbb{R}_{\geq 0};X) \cap \mathcal{C}(\mathbb{R}_{\geq 0};X_\frac{1}{2}) \cap \L^2_{\rm loc}(\mathbb{R}_{ \geq 0};X_1)$  satisfying
    \begin{equation*}
        \lVert x(t) \rVert_\frac{1}{2}^2 \leq \frac{\lambda}{2} \int_0^t \e^{- \rho(t-s)} \lVert u(s) \rVert_U^2 \d s \leq \frac{\lambda}{2 \rho} \lVert u \rVert_{\L^\infty([0,t];U)}^2.
    \end{equation*}
    Consequently, $x$ is also the mild solution to the extended system node $\Sigma(A,[\begin{smallmatrix} B & I\end{smallmatrix}],\left[\begin{smallmatrix} C\\ I\end{smallmatrix}\right])$ with input $[\begin{smallmatrix} u\\ -x^3\end{smallmatrix}]$ whose (distributional) output is $[\begin{smallmatrix} y\\ x \end{smallmatrix}]$ with $y(t)=(C\&D)\left[\begin{smallmatrix} x(t)\\ u(t)\end{smallmatrix}\right]$ for almost every $t \geq 0$  since $\left[\begin{smallmatrix} x(t)\\ u(t)\end{smallmatrix}\right] \in D(C\&D)$ for almost every $t \geq 0$. Now, as $x$ takes only values in $X_{\frac{1}{2}}=D(C)$, it suffices to show that for all $t>0$ we have that $\| C x(\cdot) \|_{\L^\infty([0,t];Y)} \lesssim \| u \|_{\L^\infty([0,t];U)} $. But this bound directly follows from the estimate of $x(t)$ in the $X_\frac{1}{2}$-norm and the boundedness of $C$ as operator from $X_{\frac{1}{2}}$ to $Y$. \qed
\end{proof}

In view of the machinery that has been considered to construct the ISS estimate \eqref{ISS_Heat}, we have the feeling that such an approach should work for nonlinearities expressed as the opposite of odd monomials as well, provided that the subscript $\alpha$ for the space $X_\alpha$ is appropriately chosen. Indeed, such nonlinear operators satisfy the sectorial condition $\langle x,f(x)\rangle_X\leq 0$, which may be viewed as a condition for the energy to be non-increasing. For instance, such a sectorial condition has already been used in \cite{HastirSCL} in order to prove the well-posedness of nonlinear infinite-dimensional systems like \eqref{eq:semilinStateSpaceNode}.

\section{An application to funnel control}\label{Section_Application_Chemical}

The applicability of funnel control for a coupled ODE--PDE system describing the evolution of chemical components in chemical reactors is studied in this section. Therefore, some notions related to funnel control are recalled in a first time. In the second part of this section, the different assumptions needed to apply funnel control are tested on the ODE--PDE system. Numerical simulations are then depicted to support the theoretical results.
We emphasize the fact that the results that we present in this section extend the ones of \cite{Hastir_Funnel} for the following reasons:
\begin{itemize}
    \item The input and output operators are allowed to be unbounded here, while their boundedness with some additional regularity assumptions are needed in \cite{Hastir_Funnel}. In particular, by denoting by $B:\mathbb{R}\to X, Bu = bu$ and $C:X\to\mathbb{R}, Cx = \langle c,x\rangle_X$, respectively, the conditions $b\in D(A)$ and $c\in D(A^*)$ are needed in \cite{Hastir_Funnel}, which may be viewed as restrictive;
    \item The assumption $\langle b,c\rangle_X\neq 0$ is required in \cite{Hastir_Funnel} in such a way that the state space $X$ can be split into two linear subspaces in order to deal with the Byrnes-Isidori forms;
    \item The nonlinearity is only allowed to be defined from the state space $X$ into itself in \cite{Hastir_Funnel}. Moreover, it is required there that the latter is globally Lipschitz continuous. More general classes of nonlinear operators are treated here.
\end{itemize}
We also note that already in the early days of funnel control, chemical reactors were studied, \cite{Ilchmann2004}.
\subsection{General considerations on funnel control for systems with relative degree one}
We recall the following framework yielding funnel control, which was already present in the early works of the field, \cite{Ilchmann2002}, see also \cite{Berger2021Bis} and the references therein.

For the following input-output differential relation
\begin{equation}
\dot{y}(t) = N(d(t), S(y)(t)) + M(d(t),S(y)(t))u(t),\,\, y(0) = y_0,
\label{OutputEquation}
\end{equation}
it is supposed that the following conditions hold.
\begin{assumption}\label{Disturb}
The disturbance $d\in \L^\infty(\mathbb{R}_{\geq 0};\mathbb{R})$, the nonlinear function $N$ is in $\mathcal{C}(\mathbb{R}^2;\mathbb{R})$ and the gain function $M\in\mathcal{C}(\mathbb{R}^2;\mathbb{R})$ is positive in the sense that $M(d,\varrho)>0$ for all $(d,\varrho)\in\mathbb{R}^2$.
\end{assumption}
\begin{assumption}\label{NonlinearMapT}
The map $S:\mathcal{C}(\mathbb{R}_{\geq 0};\mathbb{R})\to \L^\infty_{\text{loc}}(\mathbb{R}_{\geq 0};\mathbb{R})$ is a (possibly nonlinear) operator which satisfies the following conditions:
\begin{enumerate}
\item Bounded trajectories are mapped into bounded trajectories (BIBO property),\newline i.e. for all $k_1>0$, there exists $k_2>0$ such that for all $y\in\mathcal{C}(\mathbb{R}_{\geq 0};\mathbb{R})$,
\begin{equation}
\Vert y\Vert_{\L^\infty(\mathbb{R}_{\geq 0};\mathbb{R})}\leq k_1 \quad \implies \quad \Vert S(y)\Vert_{\L^\infty([0,t];\mathbb{R})}\leq k_2,
\label{BIBOAssum}
\end{equation}

\item The operator $S$ is causal, \\ i.e. for any $t\in\mathbb{R}_{\geq 0}$ and any $y,\hat{y}\in\mathcal{C}(\mathbb{R}_{\geq 0};\mathbb{R})$, 
\begin{equation*}y_{\vert_{[0,t)}} = \hat{y}_{\vert_{[0,t)}} \implies S(y)_{\vert_{[0,t)}} = S(\hat{y})_{\vert_{[0,t)}},
\end{equation*}
where $f_{\vert_I}$ denotes the restriction of a function $f$ to the interval $I$.
\item $S$ is locally Lipschitz in the following sense:\\ for all $t\in\mathbb{R}_{\geq 0}$ and all $y\in\mathcal{C}([0,t];\mathbb{R})$ there exist positive constants $\tau,\delta$ and $\rho$ such that for any $y_1,y_2\in\mathcal{C}(\mathbb{R}_{\geq 0};\mathbb{R})$ with $y_{i_{\vert [0,t]}} = y, i=1,2$ and $\vert y_i(s)-y(t)\vert < \delta$ for all $s\in[t,t+\tau]$ and $i=1,2$ it holds that
\begin{equation}
\Vert S(y_1) - S(y_2)\Vert_{\L^\infty([t,t+\tau];\mathbb{R})} \leq \rho\Vert y_1-y_2\Vert_{\L^\infty([t,t+\tau];\mathbb{R})}.
\label{Lipschitz}
\end{equation}
\end{enumerate}
\end{assumption}
The class of systems governed by \eqref{OutputEquation} with Assumptions~\ref{Disturb}--\ref{NonlinearMapT} is presented in \cite[Section~1]{BergerPucheSchwenninger} for systems with (possible) memory and relative degree $r\in\mathbb{N}$. This class is quite general and encompasses systems with infinite-dimensional internal dynamics as shown in \cite{BergerPucheSchwenninger} and \cite{IlchmannByrnesIsidori} for instance.

For systems written like in \eqref{OutputEquation}, a funnel controller is an adaptive model-free control method whose objective is to maintain the difference between the output and an a priori fixed reference signal within the following prescribed funnel $\mathcal{F}_\phi \coloneqq \left\{(t,e(t))\in\mathbb{R}_{\geq 0}\times\mathbb{R},\, \phi(t)\vert e(t)\vert < 1\right\}$, where the function $\phi$ is assumed to belong to
\begin{equation*}
\Phi \coloneqq \left\{\phi\in\mathcal{C}(\mathbb{R}_{\geq 0};\mathbb{R}) \colon \phi, \dot{\phi}\in \L^\infty(\mathbb{R}_{\geq 0};\mathbb{R}), \phi(t)> 0, \forall
t\in\mathbb{R}_{\geq 0} \text{ and } \liminf_{t\to\infty}\phi(t) > 0\right\}.
\end{equation*}
As described in \cite{BergerPucheSchwenninger,IlchmannByrnesIsidori} or \cite{BergerFunnelAutomatica}, a controller that achieves the  described output tracking performance is given by
\begin{equation}
u(t) = \frac{-e(t)}{1-\phi^2(t)e^2(t)},
\label{FunnelController}
\end{equation}
with $\phi\in\Phi$ and $\phi(0)\vert e(0)\vert < 1$. The following theorem, coming from \cite{Ilchmann2002}, see also \cite{BergerFunnelAutomatica} with $r=1$, characterizes the effectiveness of the controller \eqref{FunnelController} in terms of existence and uniqueness of solutions of the closed-loop systems and in terms of output tracking performance.
\begin{theorem}[\cite{Ilchmann2002},\cite{BergerFunnelAutomatica}]\label{ThmFunnel}
Consider System \eqref{OutputEquation} with Assumptions~\ref{Disturb}--\ref{NonlinearMapT}. Let $y_{\text{ref}}\in W^{1,\infty}(\mathbb{R}_{\geq 0};\mathbb{R}), \phi\in\Phi$ and $y_0\in\mathbb{R}$ such that the condition $\phi(0)\vert e(0)\vert < 1$ holds. Then the funnel controller \eqref{FunnelController} applied to \eqref{OutputEquation} results in a closed-loop system whose solution $y:[0,\omega)\to\mathbb{R}, \omega\in (0,\infty]$, has the following properties:
\begin{enumerate}
\item The solution is global, i.e. $\omega = \infty$;
\item The input $u:\mathbb{R}_{\geq 0}\to\mathbb{R}$, the gain function $k:\mathbb{R}_{\geq 0}\to\mathbb{R}$ and the output $y:\mathbb{R}_{\geq 0}\to\mathbb{R}$ are bounded;
\item The tracking error $e:\mathbb{R}_{\geq 0}\to\mathbb{R}$ evolves in the funnel $\mathcal{F}_\phi$ and is bounded away from the funnel boundaries in the sense that there exists $\epsilon> 0$ such that, for all $t\geq 0, \vert e(t)\vert \leq \frac{1}{\phi(t)} - \epsilon$.
\end{enumerate}
\end{theorem}

\subsection{Dynamical analysis of a coupled ODE -- PDE system for funnel control}

\begin{figure}[H]
\begin{center}
\includegraphics[scale=1]{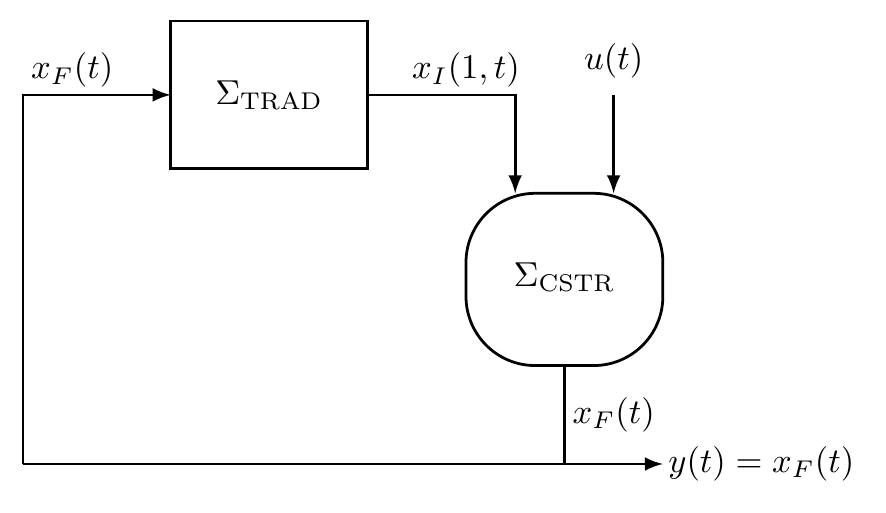}
    \caption{Coupled CSTR-tubular reactor system}
    \label{fig:chemicalSystem}
\end{center}
\end{figure}

Consider the system depicted in \Cref{fig:chemicalSystem} comprised of a continuous stirred-tank reactor (CSTR) and a tubular reactor with axial dispersion (TRAD) that is similar to the one discussed in \cite{Khatibi2020}, dropping the terms describing disturbances for simplicity. Furthermore, in order for the system being able to serve as an example for the type of systems discussed in this contribution, two changes have been made. First, the output has been chosen to be taken in the ODE part of the system in order for the system to be of relative degree one. Second, the Dirichlet boundary control has been supplanted with a Neumann boundary control ensuring that the original linear distributed system will be BIBO stable. Using the same notation as \cite{Khatibi2020} this input-output system is then described by the coupled PDE-ODE system
\begin{align*}
        &\Sigma_{\text{TRAD}}\left\{
        \begin{aligned}
        \displaystyle\frac{\partial x_I}{\partial t}(\zeta,t) &= D  \frac{\partial^2 x_I}{\partial \zeta^2}(\zeta,t) - v \frac{\partial x_I}{\partial \zeta}(\zeta,t) - \psi x_I(\zeta,t) + f(x_I(\zeta,t)),\\ 
        x_I(\zeta, 0) &= 1, \\
        \displaystyle\frac{\partial x_I}{\partial \zeta}(0,t) &= x_F(t), \,  \frac{\partial x_I}{\partial \zeta}(1,t) = 0, \\
        \end{aligned}\right.
        \\
        &\Sigma_{\text{CSTR}}\left\{
        \begin{aligned}
        \Dot{x}_F(t) &= a_1 x_F(t) + a_2 u(t) + R x_I(1, t),\\
        x_F(0) &= 1,\\
        y(t) &= x_F(t),
        \end{aligned}\right.
\end{align*}
where the temporal and the spatial variables satisfy $t\geq 0$ and $\zeta\in [0,1]$, respectively, with $x_F(t) \in \mathbb{R}$ and $x_I(\cdot,t)\in \L^2([0,1]; \mathbb{R})$. The constants $v > 0$ and $D > 0$ are the transport and diffusion velocities in the tubular reactor, $R > 0$ describes the recycling within the system, and $a_1$, $a_2$ and $\psi > 0$ are constants describing the chemical reactions within the two reactors. Furthermore $f$ is a nonlinear mapping from $\L^2([0,1]; \mathbb{R})$ to $\L^2([0,1]; \mathbb{R})$, such as e.g. the Lipschitz continuous function $f(x) = \frac{\vert x\vert}{\vert x\vert+1}$ from \cite{DelattreDochainWinkin2004}.

We can straightforwardly bring this system into the principle form of the system class in \cite{Berger2020funnel} as
\begin{equation*}
\Dot{y}(t) = S(y)(t) + a_2 \, u(t),\,\,y(0) = 1
\end{equation*}
with the operator $S: \mathcal{C}(\mathbb{R}_{\geq 0}; \mathbb{R}) \rightarrow L_{\textrm{loc}}^\infty(\mathbb{R}_{\geq 0}; \mathbb{R})$ given by
\begin{equation}
    S(\eta(\cdot)) = a_1 \eta(\cdot) + R x(1,\cdot)
    \label{Map_S}
\end{equation}
where $x$ is the solution of the system
\begin{equation}
\left\{
\begin{aligned}
        \frac{\partial x}{\partial t}(\zeta,t) &= D  \frac{\partial^2 x}{\partial \zeta^2}(\zeta,t) - v \frac{\partial x}{\partial \zeta}(\zeta,t) - \psi x(\zeta,t) + f(x)(\zeta,t),\\
        x(\zeta,0) &= 1,\\
        \frac{\partial x}{\partial \zeta}(0,t) &= \eta(t), \, \frac{\partial x}{\partial \zeta}(1,t) = 0.
        \end{aligned}\right.\label{Eq:PDE-BC_NL}
\end{equation}
While these internal dynamics are given in terms of a boundary control system, one could -- using the methods laid out in \cite[Chapter~10]{CurtainZwartNew} and \cite{SchwenningerISS} -- rewrite this system to arrive at one in the form of Equation $\eqref{eq:semilinStateSpaceNode}$.  

The state space $X = \L^2([0,1];\mathbb{R})$ is equipped with the following weighted inner product
\begin{equation*}
    \langle f,g\rangle_\rho \coloneqq \int_0^1 \rho(\zeta) f(\zeta)g(\zeta) \d \zeta,
\end{equation*}
where $\rho(\zeta) \coloneqq \e^{-\frac{v}{D}\zeta}$. The operator $A$ is defined as $Ax = D\frac{\mathrm{d}^2x}{\mathrm{d}\zeta^2} - v\frac{\mathrm{d}x}{\mathrm{d}\zeta} + \psi x$, where $x$ is in $D(A)$ expressed as 
\begin{equation*}
    D(A) \coloneqq \left\{x\in \H^2([0,1];\mathbb{R}) \colon \frac{\mathrm{d}x}{\mathrm{d}\zeta}(0) = 0 = \frac{\mathrm{d}x}{\mathrm{d}\zeta}(1)\right\}.
\end{equation*}
The boundary operator $B: \mathbb{R}\to X_{-1}$ is given as $Bu = -D\delta_0 u$, where $\delta_0$ denotes the Dirac delta distribution at $\zeta = 0$. The observation operator $C:X_1\to\mathbb{R}$ is the point measurement at $\zeta = 1$, i.e. $Cx = x(1)$. With this framework, it is easy to see that the operator $A$ is self-adjoint by considering $\langle\cdot,\cdot\rangle_\rho$ as inner product. Moreover, $A$ is a Riesz-spectral operator whose eigenvalues and normalized eigenfunctions are given by 
\begin{equation*}
    \lambda_0 = -\psi, \,\, \lambda_n = -\frac{v^2+4D^2n^2\pi^2}{4D} - \psi, n\in\mathbb{N}_0
\end{equation*}
and
\begin{align*}
    \phi_0(\zeta) &= \sqrt{\frac{D}{v(1-\e^{\frac{-v}{D}})}}1_{[0,1]}(\zeta),\,\, \phi_n(\zeta) = \frac{\sqrt{2}v}{\sqrt{4n^2\pi^2D^2+v^2}}\left[\e^{\frac{v}{2D}\zeta}\left(\sin(n\pi\zeta) -\frac{2n\pi D}{v}\cos(n\pi\zeta)\right)\right],
\end{align*}
respectively. 

In order to apply \Cref{cor:BIBO_under_global_Lip_on_X} to deduce BIBO stability of the semilinear system, one shall first check that the extended linear system $(A,[\begin{smallmatrix} B & I\end{smallmatrix}],\left[\begin{smallmatrix} C\\ I\end{smallmatrix}\right])$ is $\L^\infty$-BIBO stable.
\begin{prop}
The extended linear system $(A,[\begin{smallmatrix} B & I\end{smallmatrix}],\left[\begin{smallmatrix} C\\ I\end{smallmatrix}\right])$ is $\L^\infty$-BIBO stable.
\end{prop}
\begin{proof}
    We use \Cref{prop:extendedLinSystemBIBO} to show this.
    \begin{description}
    \item[Condition~1: $\L^\infty$-control-admissibility of $B$] \hfill \newline
Observe that any of the spaces $X_{-\eta}, 0\leq \eta\leq 1$, can be expressed as 
\begin{equation*}
    X_{-\eta} = \left\{z\in X_{-1} \colon \sum_{n=0}^\infty \frac{\vert\langle z,\phi_n\rangle_X\vert^2}{\vert -\lambda_n\vert^{2\eta}} < \infty\right\}.
\end{equation*}
Now for $z = Bu = -D \delta_0 u$ we have that $\vert\langle z,\phi_n\rangle_X\vert^2 \sim n^0$ so that
\begin{equation*}
    \sum_{n=0}^\infty \frac{\vert\langle z,\phi_n\rangle_X\vert^2}{\vert -\lambda_n\vert^{2\eta}} \sim \sum_{n=0}^\infty \frac{1}{n^{4\eta}},
\end{equation*}
which converges for all $\eta > \frac{1}{4}$. 
Hence, $-D\delta_0 u \in X_{-\eta}$ for any $\eta\in (\frac{1}{4},1]$ and thus in particular, $B\in\mathcal{L}(U,X_{-\frac{1}{2}})$ which entails that $B$ is $\L^\infty$-admissible, see~\Cref{rem:(-A)^alpha_Linfty-adm}.
    \item[Condition~2: $\L^\infty$-BIBO stability of $\Sigma(A,B,C,\mathbf{G})$] \hfill \newline
    This follows directly from \cite[Proposition~4.3]{SWZ_BIBOStability} after finding that (using the nomenclature of \cite{SWZ_BIBOStability}) we have for $n > 0$ that $b_n \sim n^0$, $c_n \sim n^0$ and $\textrm{Re}(\lambda_n) \sim n^2$ so that 
    \begin{equation*}
        \sum_{n > 0} \frac{|b_n c_b|}{|\textrm{Re}(\lambda_n)|} \sim \sum_{n > 0} \frac{1}{n^2} < \infty.
    \end{equation*}
    \item[Condition~3: $\L^\infty$-BIBO stability of $\Sigma(A,I,C,C(\cdot I - A)^{-1})$] \hfill \newline
    The boundedness of the operator $C$ from $X_\frac{1}{2}$\footnote{The space $X_{\frac{1}{2}}$ is the first order Solobev space, namely $\H^1([0,1];\mathbb{R})$ with the following inner product
\begin{equation*}
    \langle f,g\rangle_\frac{1}{2} = D\int_0^1 \e^{-\frac{v}{D}\zeta}\frac{\mathrm{d}f}{\mathrm{d}\zeta}\frac{\mathrm{d}g}{\mathrm{d}\zeta} \d \zeta + \psi\int_0^1 \e^{-\frac{v}{D}\zeta}fg \d \zeta,
\end{equation*}
which is equivalent to the standard inner product whose $X_\frac{1}{2}$ is equipped with.} into $\mathbb{R}$ directly implies this by \cite[Proposition~6.4]{SWZ_BIBOStability}.     
\qed
    \end{description}
\end{proof}

Let us now fix $\alpha = \frac{1}{2}$. According to \Cref{cor:BIBO_under_global_Lip_on_X}, the next assumption that is needed is $B\in\mathcal{L}(U,X_{-(1-\alpha)})$ with $(-A)^\alpha B$ being $\L^\infty$-admissible. Since $\alpha = \frac{1}{2}$, first observe that $-D\delta_0\in X_{-(1-\alpha)}$. Secondly, note that $(-A)^\alpha B\in\mathcal{L}(U,X_{-\beta})$ for any $\beta\in(\frac{3}{4},\frac{3}{2}]$. In particular, $(-A)^\frac{1}{2}B\in\mathcal{L}(U,X_{-\beta})$ for $\beta \in (\frac{3}{4},1)$. As a consequence of \Cref{rem:(-A)^alpha_Linfty-adm}, i., the operator $(-A)^\alpha B$ is $\L^\infty$-admissible. 

The last condition we will check in order to have $\L^\infty$-BIBO stability of the nonlinear system \eqref{Eq:PDE-BC_NL} is $\delta^{-\frac{1}{2}}LM_\frac{1}{2}\Gamma(\frac{1}{2})<1$. First, observe that the nonlinear operator $f:X\to X, f(x) = \frac{\vert x\vert}{\vert x\vert + 1}$ is globally Lipschitz continuous from $X$ into $X$. Consequently, it preserves that property when defined from $X_\frac{1}{2}$ into $X$. One of its Lipschitz constant, computed as the supremum of the derivatives of the scalar function $f:\mathbb{R}\to [0,1], f(t) = \frac{\vert t\vert}{\vert t\vert + 1}$, is given by $L = 1$. 

We shall now give an estimate on the constant $M_\frac{1}{2}$, coming from the inequality \eqref{EstimateAnalytic}. For this, let us assume that the constant $\delta$ satisfies $\delta\in[\varepsilon,\psi-\varepsilon]$ for some $\varepsilon > 0$ sufficiently small in such a way that the semigroup generated by the operator $A+\delta I$ is still analytic and exponentially stable. It can be shown that the operator norm of $(-A)^\frac{1}{2}T(t)$ is upper-bounded as follows:
\begin{align*}
    &\Vert (-A)^\frac{1}{2}T(t)\Vert\leq\frac{\e^{-\delta t}}{t^{-\frac{1}{2}}}\sqrt{\max\left\{\frac{\psi}{\psi-\delta},\frac{-\lambda_1}{-\lambda_1-\delta},\e\right\}}\frac{\e^{-\frac{1}{2}}}{\sqrt{2}}\leq\frac{\e^{-\delta t}}{t^{-\frac{1}{2}}}\sqrt{\max\left\{\frac{\psi}{\varepsilon},\e\right\}}\frac{\e^{-\frac{1}{2}}}{\sqrt{2}},
\end{align*}
where the relations $\delta\in(\varepsilon,\psi-\varepsilon)$ and $-\lambda_1 = \psi + c$ for some $c>0$ have been used. This implies that a possible constant $M_\frac{1}{2}$ is $M_\frac{1}{2} = \sqrt{\max\left\{\psi\varepsilon^{-1},\e\right\}}(2\e)^{-\frac{1}{2}}$. As a consequence, since $L=1$, $\delta^{-\frac{1}{2}} LM_{\frac{1}{2}} \Gamma(\frac{1}{2}) <1$ is equivalent to
\begin{equation}
\max\left\{\psi\varepsilon^{-1},\e\right\}<2 \e \delta\pi^{-1},
\label{AppliThm4}
\end{equation}
where $\Gamma(\frac{1}{2}) = \sqrt{\pi}$ has been used. 

By using \Cref{cor:BIBO_under_global_Lip_on_X}, one may conclude that a sufficient condition for the nonlinear system \eqref{Eq:PDE-BC_NL} to be $\L^\infty$-BIBO stable is that \eqref{AppliThm4} holds. This result shows that the map $S$, mapping the input $\eta$ to the output $x(1,\cdot)$ is $\L^\infty$-BIBO stable, which shows that the relation \eqref{BIBOAssum} of \Cref{NonlinearMapT} is satisfied.

It remains to show that the map $S$ is locally Lipschitz continuous in the sense of \eqref{Lipschitz} to be able to apply funnel control. For this, let us consider the following proposition.

\begin{prop}
Under assumption \eqref{AppliThm4}, the operator $S$ defined in \eqref{Map_S} satisfies the inequality \eqref{Lipschitz} of \Cref{NonlinearMapT}.
\end{prop}
\begin{proof}
Let us fix $t\in\mathbb{R}_{\geq 0}$ and let us consider an arbitrary $\tau\in\mathbb{R}_{\geq 0}$. Moreover, let us pick any $\eta_1, \eta_2\in\mathcal{C}(\mathbb{R}_{\geq 0};\mathbb{R})$ that satisfy $\eta_{1_{\vert_{[0,t]}}} = \eta = \eta_{2_{\vert_{[0,t]}}}$ for some fixed $\eta\in\mathcal{C}([0,t];\mathbb{R})$. As a consequence, the mild solutions of \eqref{Eq:PDE-BC_NL} with $\eta_1$ and $\eta_2$ as inputs, denoted by $x_1$ and $x_2$, respectively, are equal up to time $t$. For $\tilde{t}\in[t,t+\tau]$, these mild solutions are expressed as
\begin{align*}
    (-A)^\frac{1}{2}x_i(\tilde{t}) &= (-A)^\frac{1}{2}T(\tilde{t}-t)x_i(t) + \int_t^{\tilde{t}} T(\tilde{t}-s)(-A)^\frac{1}{2}f(x_i(s)) \d s + \int_t^{\tilde{t}} T(\tilde{t}-s)(-A)^\frac{1}{2}B\eta_i(s) \d s,
\end{align*}
$i = 1,2$. By making the difference between $x_1$ and $x_2$ and by taking the $X$-norm on both sides implies that 
\begin{align*}
    \Vert x_1 - x_2\Vert_{\L^\infty([t,t+\tau];X_\frac{1}{2})}\leq M_\frac{1}{2}L\Gamma(\tfrac{1}{2})\delta^{-\frac{1}{2}}\Vert x_1 - x_2\Vert_{\L^\infty([t,t+\tau];X_\frac{1}{2})}+C_{2,\infty}\Vert \eta_1 - \eta_2\Vert_{\L^\infty([t,t+\tau];\mathbb{R})},
\end{align*}
where $C_{2,\infty}$ and $M_\frac{1}{2}\Gamma(\frac{1}{2})\delta^{-\frac{1}{2}}$ are the infinite-time $\L^\infty$-admissibility constants of the operators, $(-A)^\frac{1}{2}B$ and $(-A)^\frac{1}{2}$, respectively. Assumption \eqref{AppliThm4} together with the boundedness of the trace operator from $X_\frac{1}{2}$ to $\mathbb{R}$ conclude the proof.\qed
\end{proof}

According to \Cref{ThmFunnel}, funnel control is conducive for \eqref{Eq:PDE-BC_NL} provided that the initial error between the output and the tracked reference is in the prescribed funnel. Some numerical simulations are reported in the following section.
\subsection{Numerical simulations}\label{Sec:numsim}
As parameters for the PDE and the ODE in \eqref{Eq:PDE-BC_NL}, we consider the following values $D = 0.1, v = 0.4, \psi = 2.8, a_1 = -1, a_2 = 2, R = 3$.
The reference signal that the output $x_F(t)$ is supposed to track is set as $y_{ref}(t) = \frac{1}{2}\cos(t)$, while the prescribed funnel boundaries in which the output error evolves are fixed to $\frac{-1}{\phi(t)}$ and $\frac{1}{\phi(t)}, \phi(t) = (2\e^{-2t} + 0.2)^{-1}$. The spatial interval $[0,1]$ is discretized into $n$ equal pieces, $n = 100$. Then the PDE -- ODE system \eqref{Eq:PDE-BC_NL} in closed-loop with the funnel controller \eqref{FunnelController} is discretized by using finite differences and it is integrated afterwards with the ODE solver \texttt{ode23s} of Matlab\copyright. The resulting state of the PDE \eqref{Eq:PDE-BC_NL}, namely $x_I(\zeta,t)$ is depicted in \Cref{fig:PDE_State}. The error between the output $x_F(t)$ and the reference signal $y_{ref}(t)$ together with the prescribed funnel are given in \Cref{fig:Error_VS_Funnel}. The funnel controller is depicted in \Cref{fig:Input}.

\begin{figure}[H]
    \centering
    \includegraphics[trim = 0cm 1.5cm 1cm 1cm,scale=0.45]{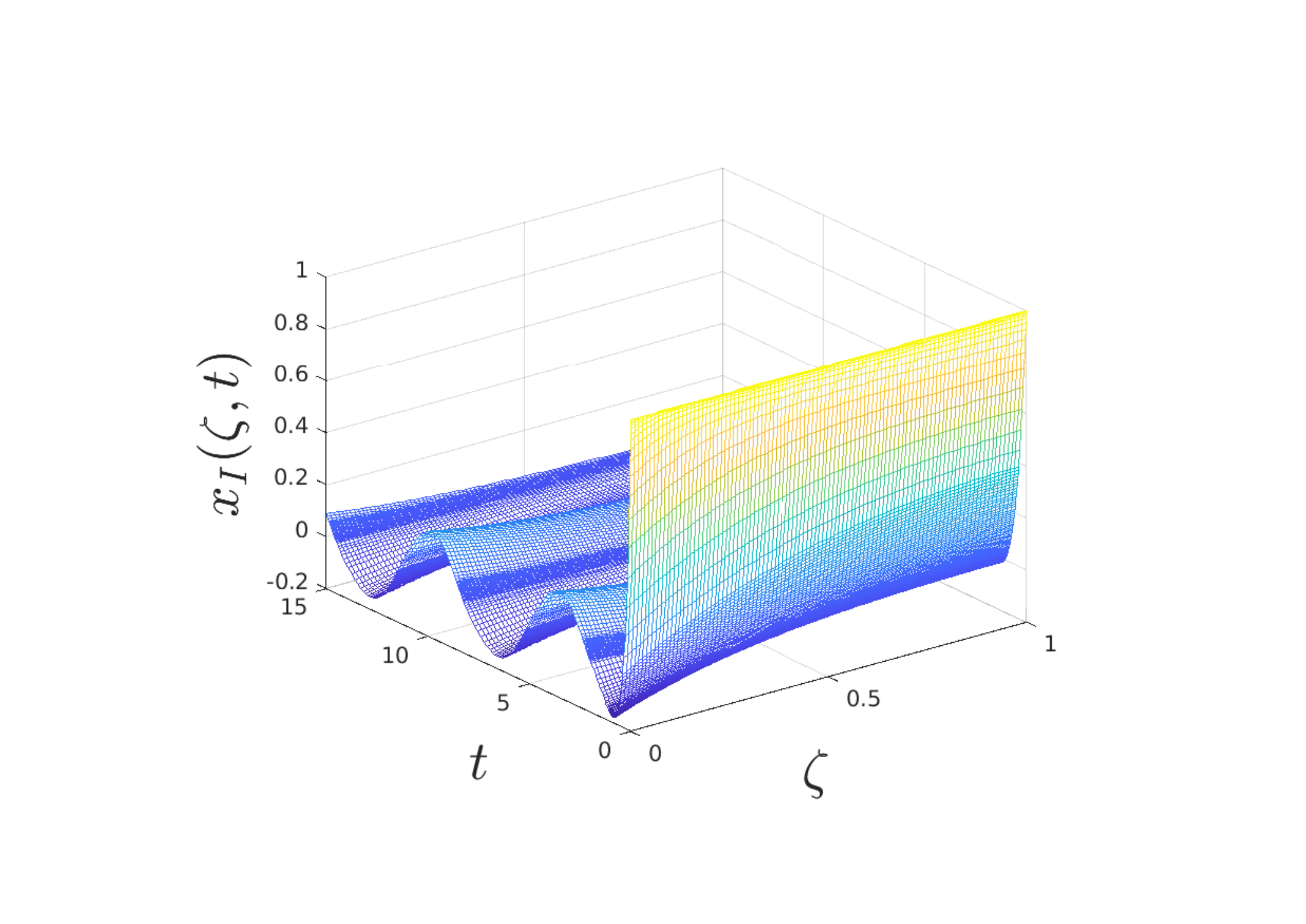}
    \caption{State variable $x_I(\zeta,t)$ of the PDE \eqref{Eq:PDE-BC_NL}.}
    \label{fig:PDE_State}
\end{figure}
\begin{figure}[H]
    \centering
    \includegraphics[trim=0cm 1.5cm 0cm 0cm,scale=0.45]{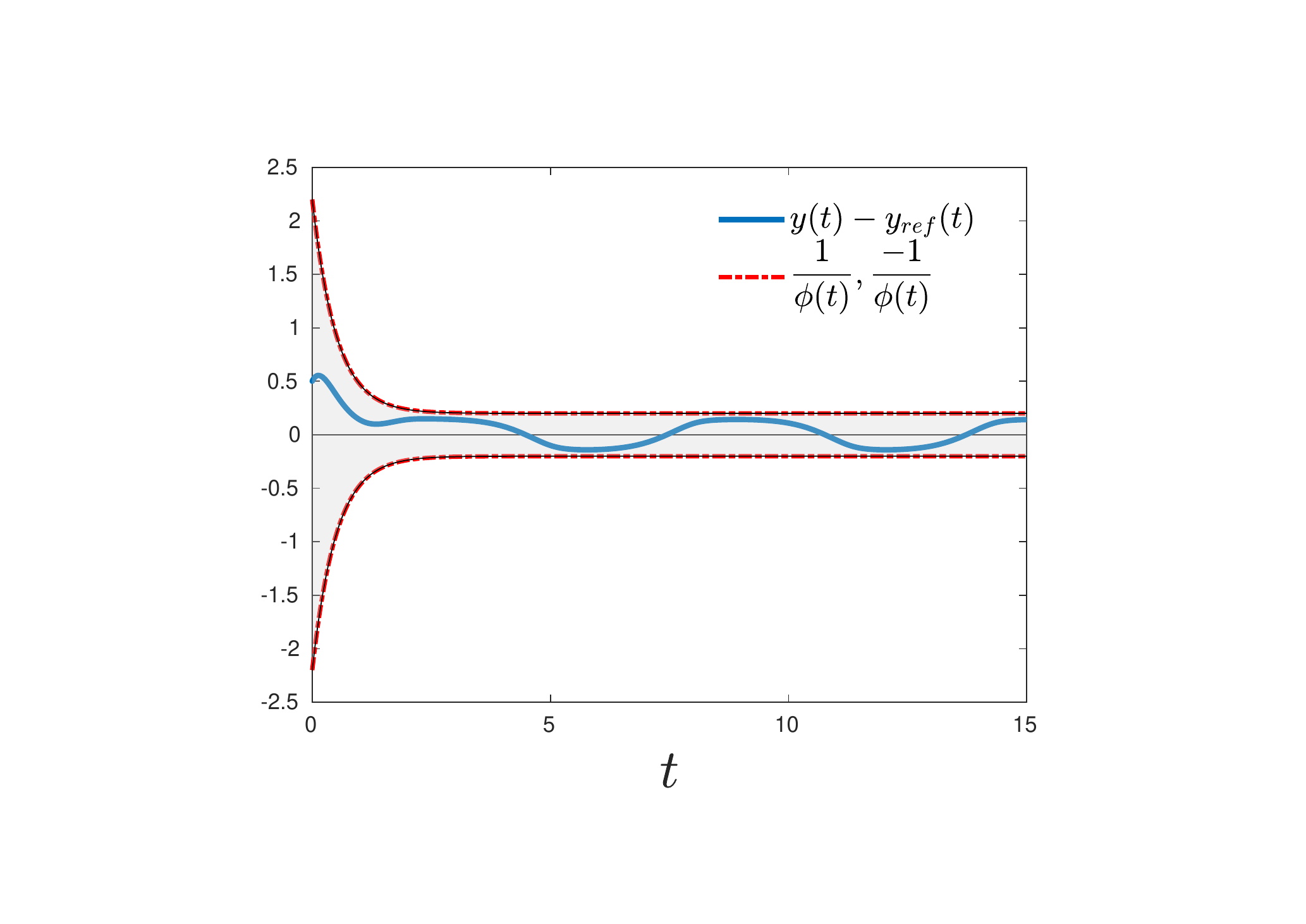}
    \caption{Output error tracking $e(t) = y(t) - y_{ref}(t)$ with the funnel whose boundaries are the functions $-\frac{1}{\phi(t)}$ and $\frac{1}{\phi(t)}$.}
    \label{fig:Error_VS_Funnel}
\end{figure}
\begin{figure}[H]
    \centering
    \includegraphics[trim=0cm 1.5cm 0cm 0cm,scale=0.45]{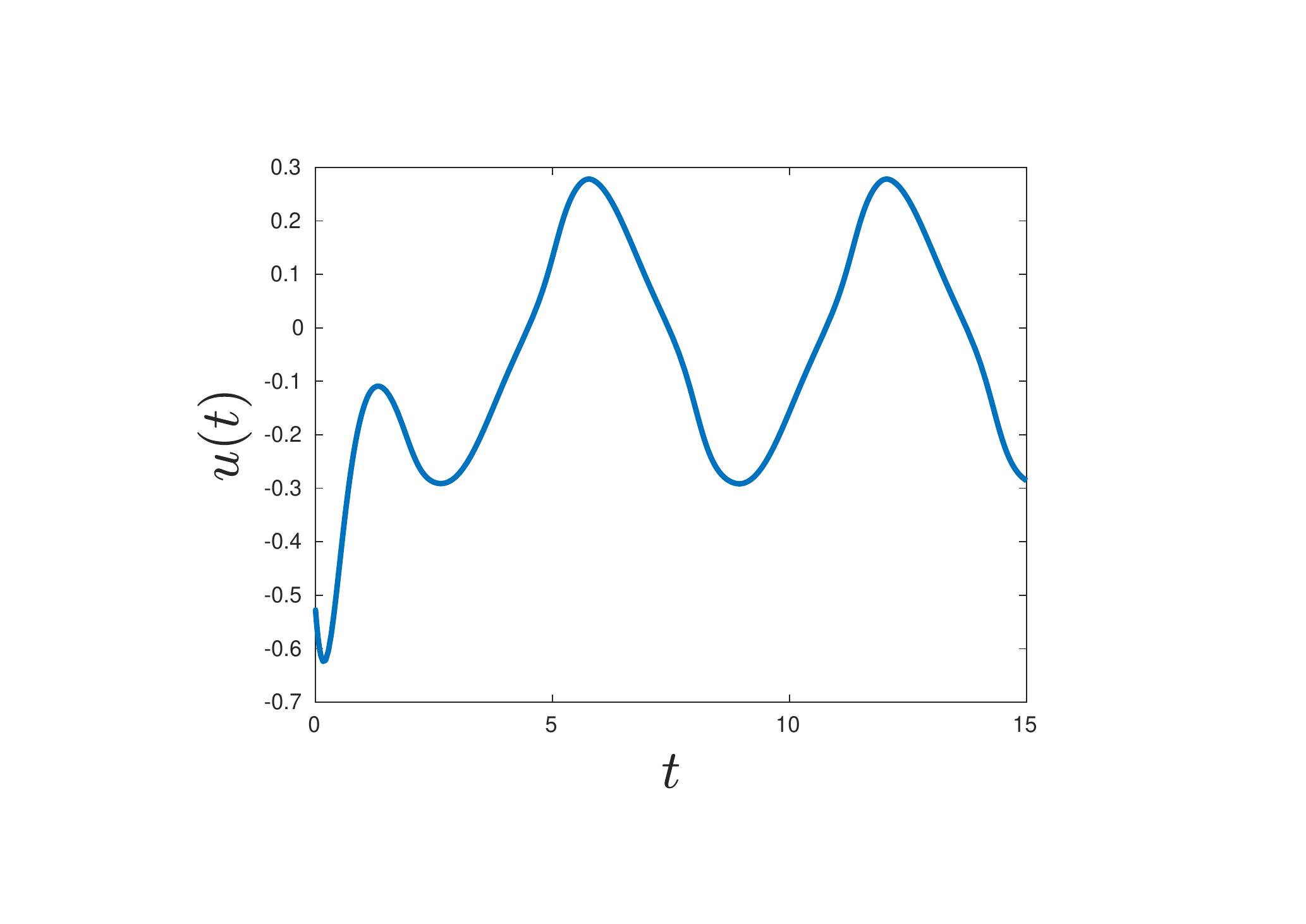}
    \caption{Funnel controller $u(t) = \frac{-e(t)}{1-\phi^2(t)e^2(t)}$.}
    \label{fig:Input}
\end{figure}

\section{Conclusion and perspectives}\label{Section_CCL}
BIBO stability for nonlinear infinite-dimensional systems with general input and output operators has been considered in this paper. By expressing the nonlinear system as the interconnection between an extended linear system and a nonlinear feedback, it has been shown how BIBO stability of the original linear system was linked to BIBO stability of the extended linear system. Then, by considering different types of globally Lipschitz nonlinear operators, $\L^\infty$-BIBO stability of the original nonlinear system has been studied. In the case where the linear operator $A$ generates an analytic semigroup, we showed that sufficient conditions for the nonlinear system to be BIBO stable are the $\L^\infty$-admissibility of the operator $(-A)^\alpha B$ for some $\alpha\in [0,1)$ together with the inequality $C_{2,\infty}L<1$, $C_{2,\infty}$ and $L$ standing for the $\L^\infty$-admissibility constant of the operator $(-A)^\alpha$ and the Lipschitz constant of the nonlinear operator, respectively. This is the main contribution of our paper. We applied our results to build a funnel controller for a nonlinear tubular reactor model coupled with a continuous stirred tank reactor with unbounded input and output operators. 

As future works, relaxing the global Lipschitz continuity assumption should be investigated. This would considerably enlarge the applicability of the results. Moreover, investigating BIBO stability in the case where the nonlinearity enters into the dynamics via an unbounded operator and in the case where an unbounded operator acts on the second component of the output $\tilde{y}$ is a challenge for further research. We finally mention that the assumption that the extended linear system is BIBO requires that the original linear system is exponentially stable, see \Cref{rem:extendedlinearsys}. This can be an unnecessary strong assumption, as the nonlinearity may in fact cause the stability rather than the linear dynamics, see \Cref{Section_Local_Lipschitz}.

\section*{Acknowledgments}
The authors are thankful to Prof. H.J. Zwart (UTwente) for valuable discussions on the subject.

\bibliographystyle{abbrv} 
\bibliography{biblio}   

\end{document}